\documentclass{amsart}
\usepackage{amsmath}
\usepackage{enumerate}
\usepackage{tikz}
\usepackage{pgfplots}
\usepackage{amssymb}
\usepackage{amsfonts}
\usepackage{bm}
\usepackage{marvosym}
\usepackage{caption}
\usepackage{subcaption}
\usepackage{mathrsfs}
\usepackage[OT2,T1]{fontenc}
\usepackage{glossaries}
\usepackage{graphicx}

\newtheorem{thm}{Theorem}[subsection]
\newtheorem{cor}[thm]{Corollary}
\newtheorem{lem}[thm]{Lemma}
\newtheorem{prop}[thm]{Proposition}

\theoremstyle{definition}
\newtheorem{defn}[thm]{Definition}

\theoremstyle{remark}
\newtheorem{rem}[thm]{Remark}

\numberwithin{equation}{subsection}
\numberwithin{figure}{section}

\newcommand{\C}{{\mathbb C}}
\newcommand{\D}{{\mathbb D}}
\newcommand{\T}{{\mathbb T}}
\newcommand{\R}{{\mathbb R}}

\newcommand{\Mop}{\mathbf{M}}
\newcommand{\Vop}{\mathbf{\Lambda}}

\newcommand{\Lop}{{\mathbf L}}

\newcommand{\Pop}{{\mathbf P}}

\newcommand{\Hop}{{\mathbf H}}

\newcommand{\calS}{\mathcal{S}}

\newcommand{\calH}{{\mathcal H}}
\newcommand{\calD}{{\mathcal D}}

\newcommand{\calB}{{\mathcal B}}

\newcommand{\calK}{{\mathcal K}}
\newcommand{\calQ}{{\mathcal Q}}

\newcommand{\calE}{{\mathcal E}}
\newcommand{\hDelta}{\varDelta}
\newcommand{\hdelta}{\eta}

\newcommand{\frakF}{\mathfrak{T}}

\newcommand{\kernel}{\mathrm{k}}

\newcommand{\diff}{{\mathrm d}}
\newcommand{\diffs}{\mathrm{ds}}
\newcommand{\diffA}{\mathrm{dA}}

\newcommand{\imag}{{\mathrm i}}
\newcommand{\Ordo}{\mathrm{O}}

\newcommand{\e}{\mathrm e}
\renewcommand{\Re}{\operatorname{Re}}
\renewcommand{\Im}{\operatorname{Im}}

\DeclareSymbolFont{cyrletters}{OT2}{wncyr}{m}{n}
\DeclareMathSymbol{\Rfun}{\beta}{cyrletters}{"17}
\DeclareFontFamily{U}{rcjhbltx}{}
\DeclareFontShape{U}{rcjhbltx}{m}{n}{<->rcjhbltx}{}
\DeclareSymbolFont{hebrewletters}{U}{rcjhbltx}{m}{n}

\DeclareMathSymbol{\aleph}{\mathord}{hebrewletters}{39}
\DeclareMathSymbol{\beth}{\mathord}{hebrewletters}{98}
\DeclareMathSymbol{\gimel}{\mathord}{hebrewletters}{103}
\DeclareMathSymbol{\daleth}{\mathord}{hebrewletters}{100}
\DeclareMathSymbol{\lamed}{\mathord}{hebrewletters}{108}
\DeclareMathSymbol{\mem}{\mathord}{hebrewletters}{109}
\DeclareMathSymbol{\ayin}{\mathord}{hebrewletters}{96}
\DeclareMathSymbol{\tsadi}{\mathord}{hebrewletters}{118}
\DeclareMathSymbol{\qof}{\mathord}{hebrewletters}{113}
\DeclareMathSymbol{\shin}{\mathord}{hebrewletters}{152}
\DeclareMathSymbol{\memschloss}{\mathord}{hebrewletters}{77}
\DeclareMathSymbol{\nunlange}{\mathord}{hebrewletters}{78}
\DeclareMathSymbol{\vav}{\mathord}{hebrewletters}{79}
\DeclareMathSymbol{\tet}{\mathord}{hebrewletters}{84}
\DeclareMathSymbol{\tsadiklange}{\mathord}{hebrewletters}{90}
\DeclareMathSymbol{\He}{\mathord}{hebrewletters}{104}
\DeclareMathSymbol{\kaf}{\mathord}{hebrewletters}{107}
\DeclareMathSymbol{\nun}{\mathord}{hebrewletters}{110}
\DeclareMathSymbol{\pei}{\mathord}{hebrewletters}{112}
\DeclareMathSymbol{\resh}{\mathord}{hebrewletters}{114}
\DeclareMathSymbol{\samekh}{\mathord}{hebrewletters}{115}
\DeclareMathSymbol{\Het}{\mathord}{hebrewletters}{116}
\DeclareMathSymbol{\vav}{\mathord}{hebrewletters}{119}
\DeclareMathSymbol{\het}{\mathord}{hebrewletters}{120}
\DeclareMathSymbol{\yod}{\mathord}{hebrewletters}{121}
\DeclareMathSymbol{\zayin}{\mathord}{hebrewletters}{122}
\newcommand{\indset}{\Het}
\newcommand{\indsett}{\tsadiklange}

\makeglossaries

\begin{document}

\title[Off-spectral analysis of Bergman kernels]
{Off-spectral analysis of Bergman kernels}

\author[Hedenmalm]{Haakan Hedenmalm}

\address{Hedenmalm: Department of Mathematics
\\
The Royal Institute of Technology
\\
S -- 100 44 Stockholm
\\
SWEDEN}

\email{haakanh@math.kth.se}

\author[Wennman]{Aron Wennman}

\address{Wennman: Department of Mathematics
\\
The Royal Institute of Technology
\\
S -- 100 44 Stockholm
\\
SWEDEN}

\email{aronw@math.kth.se}



\date{\today}
\begin{abstract}
The asymptotic analysis of Bergman kernels with respect to exponentially 
varying measures near emergent interfaces has attracted recent  
attention. Such interfaces typically occur when the associated limiting 
Bergman density function vanishes on a portion of the plane, 
\emph{the off-spectral region}. This type of behavior is observed when
the metric is negatively curved somewhere, or when we study partial Bergman 
kernels in the context of positively curved metrics. 
In this work, we cover these two situations in a unified way, for 
exponentially varying weights on the complex plane. 
We obtain a uniform asymptotic expansion of the
{\em coherent state of depth} $n$ rooted at an off-spectral point, 
which we also refer to as the {\em root function} at the point in
question.
The expansion is valid in the entire off-spectral 
component containing the root point, and protrudes into the spectrum as well. 
This allows us to obtain error function transition behavior of the density 
of states along the smooth interface. Previous work on asymptotic 
expansions of Bergman kernels is typically local, and valid only in the 
bulk region of the spectrum, which contrasts with our non-local expansions.  
\end{abstract}
\maketitle

\section{Introduction}
\subsection{Bergman kernels and emergent interfaces}
This article is a companion to our recent work \cite{HW} on the structure
of planar orthogonal polynomials. 
We will make frequent use of methods developed there, and recommend that
the reader keep that article available for ease of reference.

The study of Bergman kernel asymptotics has by now a sizeable literature.
The majority of the contributions have the flavor of local asymptotics 
near a given point $w_0$, under a positive curvature condition. 
However, in the study of partial Bergman kernels for the subspace
of all functions vanishing to a given order at the point $w_0$, the assumption
of vanishing has the effect of introducing a negative point mass for the 
curvature form at $w_0$.
In addition to the negative curvature which comes from considering partial
Bergman kernels defined by vanishing, we allow for the direct
effect of patches of negatively curved geometry. 
Around the set of negative curvature, a {\em forbidden region}
(or {\em off-spectral set}) emerges. This forbidden region is typically larger
than the set of actual negative curvature, and may consist of several 
connectivity components. 
Recently, the asymptotic behavior of Bergman kernels near the 
interface at the edge of the forbidden region has attracted considerable 
attention. 
In this work, we intend to investigate this in the fairly general setting of 
exponentially varying weights in the complex plane. 
The restriction of the Bergman kernel to the diagonal gives us the density of 
states, which drops steeply at the interface.
Indeed, in the forbidden region the density of states vanishes asymptotically, 
with exponential decay. One of our main results is that
the density of states across the interface converges to the error function 
in a blow-up, provided the interface is smooth. 

The key to obtaining the above-mentioned result is in fact our main result. 
It concerns the {\em expansion of the coherent state 
$\kernel_n(z,w_0)$ of depth $n$ at a given 
off-spectral point $w_0$}. This is the renormalized reproducing kernel 
function at the point $w_0$ for the Bergman space defined by vanishing 
to order $n$ at $w_0$.
When $n=0$, the coherent state of depth $0$ is just the normalized 
Bergman kernel 
$K(w_0,w_0)^{-\frac12}K(z,w_0)$. An important feature of the asymptotic expansion
is that $z$ and $w_0$ are allowed to be macroscopically separated. Such 
truly off-diagonal expansions have, to the best of our knowledge, not appeared 
elsewhere. 
It is easy to see that the coherent state of order $n=0,1,2,\ldots$,  
at $w_0$ 
form an orthonormal basis for the Bergman space, which gives an 
expansion of the density of states,
which leads to the error function asymptotics. Similarly, if we want to 
handle partial Bergman kernels
given by vanishing to order $n_0$ at $w_0$, we expand in the basis 
given by the coherent states of depth $n= n_0,n_0+1,\ldots$.

\subsection{Coherent states and elementary potential theory}
We now introduce the objects of study. The Bergman space $A^2_{mQ}$ is
defined as the collection of all entire functions $f$ in with finite 
weighted $L^2$-norm
\[
\|f\|_{mQ}^2:=\int_{\C}\lvert f(z)\rvert^2\e^{-2mQ(z)}\diffA(z)<+\infty,
\]
where $\diffA$ denotes the planar area element normalized so that the unit disk 
$\D$ has unit area, and where $Q$ is a potential with certain growth and 
regularity properties (see Definition~\ref{def:adm}).
We denote the reproducing kernel for $A^2_{mQ}$ by $K_m$, and for a given point
$w_0\in\C$, we consider the coherent state (normalized Bergman kernel)
\begin{equation}\label{eq:normalized-ker}
\kernel_{m,w_0}(z):=K_{m}(w_0,w_0)^{-\frac12} K_{m}(z,w_0),
\end{equation}
which has norm $1$ in $A^2_{mQ}$. There is a notion of the \emph{spectrum} 
$\calS$, also called the \emph{spectral droplet}.  This is the closed set 
defined in terms of the following obstacle problem. 
Let $\mathrm{SH}(\C)$ denote the cone of all subharmonic functions
on the plane $\C$, and consider the function
\[
\hat Q(z):=\sup\big\{q(z)\,:\,\, q\in\mathrm{SH}(\C),\,\,\text{and}\,\,\,
q\le Q \,\,\,\text{on}\,\,\,\C\big\}.
\]
Whenever $Q$ is $C^{1,1}$-smooth and has some modest growth at infinity, 
it is known that $\hat Q\in C^{1,1}$ as well,
and it is a matter of definition that $\hat Q\le Q$ pointwise (see, e.g.,
\cite{HM}).
Here, $C^{1,1}$ denotes the standard smoothness class of differentiable 
functions with Lipschitz continuous first order partial derivatives.
We define the \emph{spectrum} (or the \emph{spectral droplet}) as the contact 
set
\begin{equation}
\calS:=\big\{z\in\C\,:\,\hat Q(z)=Q(z)\big\}.
\label{eq:calS}
\end{equation}
The {\em forbidden region} is the complement $\calS^c=\C\setminus \calS$.

We will need these notions in the context of partial Bergman kernels as well.
For a non-negative integer $n$ and a point $w_0\in\C$, 
we consider the subspace $A^2_{mQ,n,w_0}$ of $A^2_{mQ}$, consisting of those 
functions that vanish to order at least $n$ at $w_0$. It may happen for
some $n$ that this space is trivial, for instance when the potential $Q$
has logarithmic growth only, because then the space $A^2_{mQ}$ consists 
of polynomials of a bounded degree.
We denote its reproducing kernel by $K_{m,n,w_0}$, and observe that 
$K_{m,0,w_0}=K_m$. We shall need also the {\em coherent state 
of depth $n$ at $w_0$} (or the {\em root function} of order $n$), 
denoted $\kernel_{m,n,w_0}$, which is the unique solution to the optimization 
problem
\[
\max\big\{\Re f^{(n)}(w_0):\,\,f\in A^2_{mQ,n,w_0},\,\,\|f\|_{mQ}\le1\big\}, 
\]
provided the maximum is positive, in which case the optimizer has 
norm $\|f\|_{mQ}=1$. In the remaining case, the maximum equals
$0$, and either only $f=0$ is possible, or there are several competing
optimizers, simply because we may multiply the function by unimodular 
constants and obtain alternative optimizers. In both remaining instances 
we declare that $\kernel_{m,n,w_0}=0$.
When nontrivial, the root function (or coherent state) of order $n$ at $w_0$ is 
connected with the reproducing kernel $K_{m,n,w_0}$:
\begin{equation}
\label{eq:normalized-parker}
\kernel_{m,n,w_0}(z)=
\lim_{\zeta\to w_0}K_{m,n,w_0}(\zeta,\zeta)^{-1/2}K_{m,n,w_0}(z,\zeta),
\end{equation}
where the point $\zeta$ should approach $w_0$ not arbitrarily but in a 
fashion such that the limit exists and has positive $n$-th derivative at $w_0$. 
The root function $\kernel_{m,n,w_0}$ will play a key role in our analysis, 
similar to that of the orthogonal polynomials in the context of polynomial 
Bergman kernels. The root functions $\kernel_{m,n,w_0}$ 
all have norm equal to $1$
in $A^2_{mQ}$, except when they are trivial and have norm $0$. 
As a result of the relation \eqref{eq:normalized-parker},
we may alternatively call the root function $\kernel_{m,n,w_0}$ a 
\emph{normalized partial Bergman kernel}. 
The spectral droplet 
associated to a family of partial Bergman kernels of the above type is 
defined in Subsection~\ref{ss:obst} in terms of an obstacle problem, and
we briefly outline how this is done.
For $0\le\tau<+\infty$, let $\mathrm{SH}_{\tau,w_0}(\C)$ denote the convex set
\[
\mathrm{SH}_{\tau,w_0}(\C)=\big\{q\in\mathrm{SH}\,(\C)\,:\, 
q(z)\le\tau\log\lvert z-w_0\rvert+\Ordo(1)\,\;\text{as }\, z\to w_0\big\},
\]
so that for $\tau=0$ we recover $\mathrm{SH}(\C)$. We consider the 
corresponding obstacle problem
\begin{equation}
\label{eq:obst-function-tau}
\hat{Q}_{\tau,w_0}(z)=\sup\big\{q(z)\;:\; q\in\mathrm{SH}_{\tau,w_0}(\C),\, 
q\le Q\,\,\text{on}\,\,\C\big\},
\end{equation}
and observe that $\tau\mapsto\hat{Q}_{\tau,w_0}$ is monotonically decreasing
pointwise.
We define a family of spectral droplets as the coincidence sets 
\begin{equation}
\calS_{\tau,w_0}=\big\{z\in\C\,:\,Q(z)=\hat{Q}_{\tau,w_0}(z)\big\}.
\label{eq:calStau}
\end{equation}
Due to the monotonicity, the droplets $\calS_{\tau,w_0}$ get smaller as $\tau$
increases, starting from $\calS_{0,w_0}=\calS$ for $\tau=0$. 
The \emph{partial Bergman density} 
\[
\rho_{m,n,w_0}(z):=m^{-1}K_{m,n,w_0}(z,z)\,\e^{-2mQ(z)},\qquad z\in\C,
\]
may be viewed as the normalized local dimension of the space 
$A^2_{mQ,n,w_0}(\C)$, and, in addition, it has the interpretation as the 
intensity of a corresponding (possibly infinite) Coulomb gas. In the
case $n=0$ we omit the word ``partial'' and speak of the 
\emph{Bergman density}. It is known
that in the limit as $m,n\to+\infty$ with $n=m\tau$,
\[
\rho_{m,n,w_0}(z)\to 2\hDelta Q(z)\,1_{\calS_{\tau,w_0}}(z),
\]
in the sense of convergence of distributions. In particular, $\hDelta Q\ge0$ 
holds a.e. on $\calS$. Here, we write $\hDelta$ for differential operator
$\partial\bar\partial$, which is one quarter of the usual Laplacian.
The above convergence reinforces our understanding
of the droplets $\calS_{\tau,w_0}$ as spectra, in the sense that the Coulomb 
gas may be thought to model eigenvalues (at least in the finite-dimensional 
case). 
The \emph{bulk} of the spectral droplet $\calS_{\tau,w_0}$ is the set
\[
\{z\in\mathrm{int}(\calS_{\tau,w_0}):\,
\hDelta Q(z)>0\},
\]  
where ``int'' stands for the operation of taking the interior.

\subsection{Further background on Bergman kernel expansions}
Our motivation for the above setup originates with the theory of random 
matrices, specifically the {\em random normal matrix ensembles}. 
We should mention that an analogous situation occurs in the study of 
complex manifolds. 
The Bergman kernel then appears in the study of spaces of 
$L^2$-integrable global holomorphic sections of $L^m$, where $L^m$ is a high 
tensor power of a holomorphic line bundle $L$ over the manifold,
endowed with an hermitian fiber metric $h$. 
If $\{U_i\}_i$ is a coordinate system on a manifold, 
then a holomorphic section $s$ to a line bundle $L$ can be written in 
local coordinates as $s=s_i e_i$, where $e_i$ are local basis elements for $L$ 
and $s_i$ are holomorphic functions. 
The pointwise norm of a section $s$ of $L^m$ may be then be written as 
$\lvert s\rvert_{h^m}^2=\lvert s_i\rvert^2\e^{-m\phi_i}$ on $U_i$, 
for some smooth real-valued functions $\phi_i$.
Along with a volume form on the base manifold, this defines an $L^2$ space 
which shares many characteristics with the spaces considered here.

The asymptotic behavior of Bergman kernels has been the subject of intense 
investigation. However, the understanding has largely been limited to the 
analysis of the kernel inside the bulk of the spectrum, in which case the 
kernel enjoys a full {\em local} asymptotic expansion. 
The pioneering work on Bergman kernel asymptotics begins with the efforts 
by H\"ormander \cite{Horm} and Fefferman \cite{Fefferman}. Developing
further the microlocal approach of H\"ormander, Boutet de Monvel and
Sj\"ostrand \cite{BdMSjostrand} obtain a near-diagonal expansion of the 
Bergman kernel close to the boundary of the given domain.
Later, in the context of K\"ahler geometry, the influential 
{\em peak section method} was introduced by Tian \cite{Tian}. 
His results were refined further by Catlin and Zelditch 
\cite{Catlin, Zelditch}, while the connection with microlocal analysis was 
greatly simplified in the more recent work by Berman, Berndtsson, and 
Sj{\"o}strand \cite{BBS}.
A key element of all these methods is that the kernel is determined by the
local geometry around the given point.
This property is absent when we consider the kernel near an off-spectral point
or near a boundary point of the spectral droplet. 

In the recent work \cite{HW}, we analyze the boundary behavior of polynomial
Bergman kernels, for which the corresponding spectral droplet is compact,
connected, and has a smooth Jordan curve as boundary.  
The analysis takes the path via a full asymptotic expansion of the 
orthogonal polynomials, valid off a sequence of increasing compacts which
eventually fill the droplet. 
By expanding the polynomial kernel in the orthonormal basis provided by the 
orthogonal polynomials, the error function asymptotics emerges along smooth 
spectral boundaries. 

The appearance of an interface for partial Bergman kernels in higher 
dimensional settings 
and in the context of complex manifolds has been observed more than once, 
notably in the work by Shiffman and Zelditch \cite{Shiffman-Zelditch} 
and by Pokorny and Singer \cite{pokorny-singer}. 
That the error function governs the transition behavior across the interface
was observed later in several contexts. For instance, 
in \cite{ross-singer}, Ross and Singer investigate the partial Bergman 
kernels associated to spaces of holomorphic sections vanishing along a 
divisor, and obtain error function transition behavior under the assumption 
that the set-up is invariant under a holomorphic $S^1$-action. This result was 
later extended by Zelditch and Zhou \cite{ZZ2}, in the context of 
$S^1$-symmetry.
More recently, Zelditch and Zhou \cite{ZZ} obtain the same transition 
for so-called {\em spectral partial Bergman kernels}, defined in terms of 
the Toeplitz quantization of a general smooth Hamiltonian. 

Our methods for obtaining asymptotic expansions of coherent states centered at 
off-spectral points do not easily extend to the higher complex dimensional 
setting. It appears plausible, however, that our analysis would apply when 
the complex plane $\C$ is replaced by a compact Riemann surface with one 
or possibly several punctures.

\begin{figure}
\centering
  \includegraphics[width=.7\linewidth]{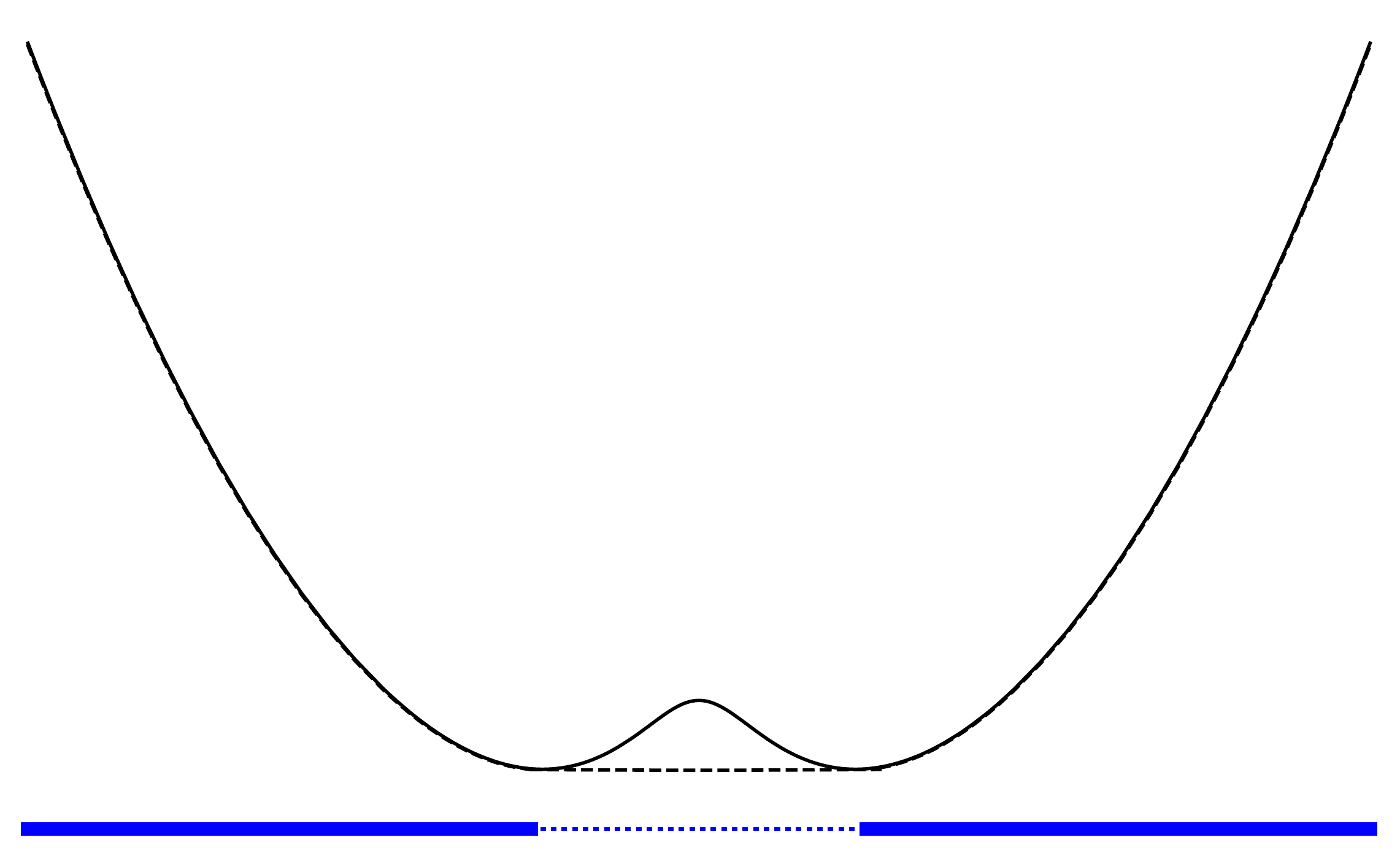}
\caption{
Illustration of the spectral droplet corresponding to the potential 
$Q(z)=\lvert z\rvert^2-\log(a+\lvert z\rvert^2)$, with $a=0.04$. 
The spectrum is illustrated with a thick line, 
and appears as the contact set between $Q$ (solid) and the solution 
$\hat{Q}$ to the 
obstacle function (dashed).}
\label{fig:obst-sing}
\end{figure}

\subsection{Off-spectral and off-diagonal asymptotics of 
coherent states}
The main contribution of the present work, put in the planar context, is a
non-local asymptotic expansion of coherent states centered at an 
off-spectral point. 
For ease of exposition, we begin with a 
version that requires as few prerequisites as possible for the formulation.
We denote by $Q(z)$ an \emph{admissible potential}, by which we mean the 
following:

\begin{enumerate}[(i)]
\item 
$Q:\C\to\R$ is $C^2$-smooth, and has sufficient growth at infinity:
\[
\tau_Q:=\liminf_{\lvert z\rvert\to+\infty}\frac{Q(z)}{\log \lvert z\rvert}>0.
\]

\item $Q$ is real-analytically smooth and strictly subharmonic in a 
neighborhood of $\partial\calS$, where $\calS$ is the contact set of 
\eqref{eq:calS},

\item there exists a bounded component $\Omega$ of the complement $\calS^c
=\C\setminus\calS$ which is simply connected, and has real-analytically 
smooth Jordan curve boundary.
\end{enumerate}

We consider the case when there exists a non-trivial off-spectral component 
$\Omega$ which is bounded and simply connected, with real-analytic boundary, 
and pick a ``root point'' $w_0\in\Omega$. To be precise, by an 
\emph{off-spectral component} we mean a connectivity component of the 
complement $\calS^c$.
This situation occurs, e.g., if the potential is strictly superharmonic 
in a portion of the plane, as is illustrated in 
Figure~\ref{fig:obst-sing}. In terms of the metric, this means that there
is a region where the curvature is negative.

\medskip

\noindent {\sc A word on notation.} To formulate our first main result, 
we need the function $\calQ_{w_0}$, which is bounded and holomorphic in the
off-spectral component $\Omega$ and whose real part equals $Q$ 
along the boundary $\partial\Omega$. To fix the imaginary part, we
require that $\calQ_{w_0}(w_0)\in\R$. In addition, we need the conformal 
mapping $\varphi_{w_0}$ which takes $\Omega$ onto the unit disk $\D$ with
$\varphi_{w_0}(w_0)=0$ and $\varphi_{w_0}'(w_0)>0$.
Since the boundary $\partial\Omega$ is assumed to be a real-analytically 
smooth Jordan curve, both the function $\calQ_{w_0}$ and the conformal mapping 
$\varphi_{w_0}$ extend analytically across $\partial\Omega$ to a fixed 
larger domain.
By possibly considering a smaller fixed larger domain, we may assume that 
the extended function 
$\varphi_{w_0}$ is conformal on the larger region.  
These observations are essential for our first main result, since we want the 
asymptotics to hold across the interface $\partial\Omega$.

\begin{thm}\label{thm:main-prel}
Assuming that $Q$ is an admissible potential, we have the following.
Given a positive integer $\kappa$ and a positive real $A$, there exist a
neighborhood $\Omega^{\circledast}$ of the closure of $\Omega$ and 
bounded holomorphic functions $\calB_{j,w_0}$ on $\Omega^{\circledast}$ 
for $j=0,\ldots,\kappa$, as well as domains $\Omega_{m}=\Omega_{m,A}$ with
$\Omega\subset\Omega_{m}\subset\Omega^{\circledast}$ which meet
\[
\mathrm{dist}_{\C}(\partial\Omega_m,\partial\Omega)\ge 
A\,m^{-\frac12}(\log m)^{\frac12},
\]
such that the normalized Bergman kernel at the point $w_0$ enjoys the 
asymptotic expansion
\begin{multline*}
\kernel_m(z,w_0)=\frac{K_m(z,w_0)}{K_m(w_0,w_0)^{1/2}}
\\
=m^{\frac14}(\varphi_{w_0}'(z))^\frac12\e^{m\calQ_{w_0}(z)}
\bigg\{\sum_{j=0}^{\kappa}m^{-j}\calB_{j,w_0}(z)+\Ordo\big(m^{-\kappa-1}\big)\bigg\},
\end{multline*}
as $m\to+\infty$, where the error term is uniform on 
$\Omega_m$. Here, the main term $\calB_{0,w_0}$ is obtained as the unique
zero-free holomorphic function on $\Omega$ which is smooth up to the boundary
with $\calB_{0,w_0}(w_0)>0$, and with prescribed boundary modulus  
\[
\lvert \calB_{0,w_0}(z)\rvert=\pi^{-\frac14}
[\hDelta Q(z)]^{\frac14},\qquad z\in\partial\Omega.
\]
Moreover, if $A$ is big enough, then
\[
\int_{\Omega_m}|\kernel_{m,w_0}(z)|^2\e^{-2mQ(z)}\diffA(z)=1+\Ordo(m^{-\kappa}).
\]
\end{thm}

As an illustration of this result, we show the Gaussian wave character of
the Berezin density $\kernel_m(w_0,w_0)^{-1}|\kernel_m(z,w_0)|^2\e^{-2mQ(z)}$ 
in Figure \ref{fig:wave}.

\begin{figure}
\centering
  \includegraphics[width=.7\linewidth]{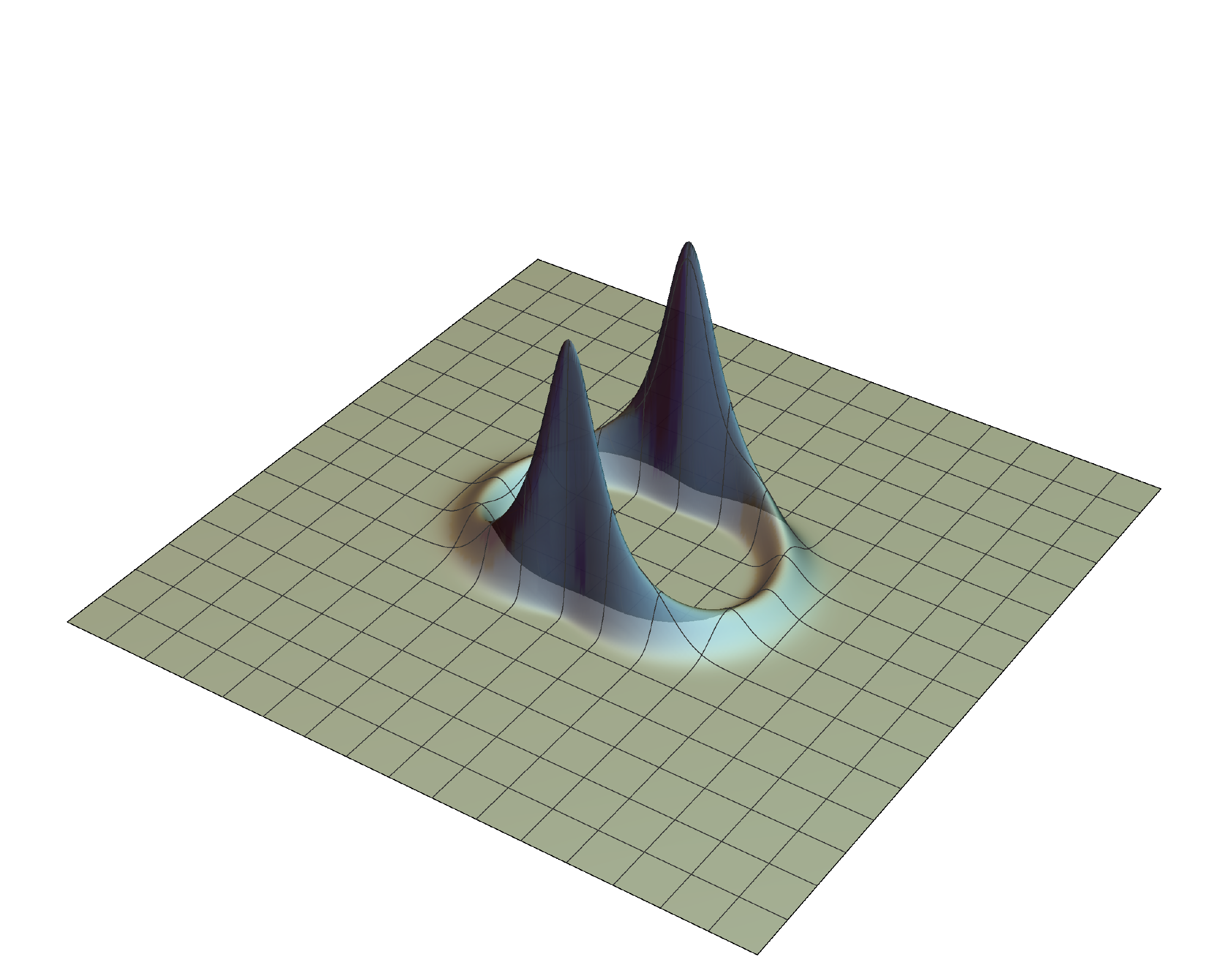}
\caption{
Illustration of the the Gaussian wave asymptotics of the probability 
density associated to 
the coherent state $\kernel_{m,0,0}(z)$, where 
$Q(z)=\frac12|z|^{-2}-\frac18\Re(z^{-2})+\log|z|$.
The root point is at the origin, and the asymptotics of 
Theorem~\ref{thm:main-prel} are 
valid in the domain $\Omega$ (inside the ridge) including a neighborhood 
of the boundary 
$\partial\Omega$ of width proportional to $(m^{-1}\log m)^{\frac12}$.}
\label{fig:wave}
\end{figure}

\begin{rem} Using an approach based on Laplace's method, the 
functions $\calB_{j,w_0}$ may be obtained algorithmically, for 
$j=1,2,3,\ldots$, see Theorem~\ref{thm:comput-coeff} below. The
details of the algorithm are analogous with the case of the orthogonal 
polynomials presented in \cite{HW}. 
\end{rem}

\noindent {\sc Commentary to the theorem.}
The point $w_0$ may be chosen as any fixed point in $\Omega$, while 
the point $z$ is allowed to vary anywhere inside the set $\Omega_m$, which 
contains the entire
off-spectral component $\Omega$ as well as a shrinking neighborhood of 
the boundary $\partial \Omega$.
Hence, we obtain {\em macroscopically off-diagonal asymptotics}, where 
the points $z$ and $w_0$ are either both off-spectral, or where $z$ is 
spectral near-boundary and $w_0$ off-spectral, respectively. 
To our knowledge, this is the first instance of such asymptotics. 
Indeed, earlier work covers either diagonal 
or near-diagonal asymptotics inside the bulk of the spectrum. The analysis 
of Bergman kernel asymptotics for two macroscopically separated points in the 
bulk $\calS^\circ$ appears difficult if we ask for high precision, 
and the same can be said for the case when $w_0\in\Omega$ is off-spectral 
and $z\in\calS^\circ$ is a bulk point. To address the latter issue, we may 
ask what happens in Theorem~\ref{thm:main-prel} if $z$ is not in the set 
$\Omega_m$.  
Since $\Omega_m$ captures most of the $L^2$-mass of the root function 
in view of Theorem~\ref{thm:main-prel}, we see that the coherent state is 
minuscule outside $\Omega_m$ 
in the $L^2$-sense. If we want corresponding pointwise control, we may 
appeal to e.g. the subharmonicity estimate of Lemma~ 3.2 of \cite{ahm1}.

\subsection{Expansion of partial Bergman kernels in terms 
of root functions}
For a $(\tau,w_0)$-admissible potential, the partial Bergman kernel 
$K_{m,n,w_0}$ with the root point $w_0$ is well-defined and nontrivial. 
In analogy with Taylor's formula, it enjoys an expansion in terms of the 
root functions $\kernel_{m,n',w_0}$ for $n'\ge n$. 

\begin{thm} 
Under the above assumption of $(\tau,w_0)$-admissibility of $Q$,
we have that
\[
K_{m,n,w_0}(z,w)=\sum_{n'=n}^{+\infty}\kernel_{m,n',w_0}(z)
\overline{\kernel_{m,n',w_0}(w)},\qquad(z,w)\in\C\times\C.
\]
\label{thm:expansion}
\end{thm}

\begin{proof}
For $n',n''\ge n$ with $n''<n'$, the functions $\kernel_{m,n',w_0}$
and $\kernel_{m,n'',w_0}$ are orthogonal in $A^2_{mQ}$. If one of them is trivial,
orthogonality is immediate, while if both are nontrivial, we argue as 
follows. Let $\zeta',\zeta''\in\C$ be close to $w_0$, and calculate that
\begin{multline*}
K_{m,n',w_0}(\zeta',\zeta')^{-\frac12}K_{m,n'',w_0}(\zeta'',\zeta'')^{-\frac12}
\big\langle K_{m,n',w_0}(\cdot,\zeta'),K_{m,n'',w_0}(\cdot,\zeta'')
\big\rangle_{mQ}
\\
=K_{m,n',w_0}(\zeta',\zeta')^{-\frac12}K_{m,n'',w_0}(\zeta'',\zeta'')^{-\frac12}
K_{m,n',w_0}(\zeta'',\zeta')=\Ordo(|\zeta''-w_0|^{n'-n''}),
\end{multline*}
which tends to $0$ as $\zeta''\to w_0$. The claimed orthogonality follows. 
Moreover, since the root functions $\kernel_{m,n',w_0}$ have unit norm 
when nontrivial, the expression 
\[
\sum_{n'=n}^{+\infty}\kernel_{m,n',w_0}(z)
\overline{\kernel_{m,n',w_0}(w)}
\]
equals the reproducing kernel function for the Hilbert space with the norm 
of $A^2_{mQ}$ spanned by the vectors $\kernel_{m,n',w_0}$ with $n'\ge n$. 
It remains to check that this is the whole partial Bergman space 
$A^2_{mQ,n,w_0}$. To this end, let $f\in A^2_{mQ,n,w_0}$ be orthogonal to all the
the vectors $\kernel_{m,n',w_0}$ with $n'\ge n$. By the definition of the space
$A^2_{mQ,n,w_0}$, this means that $f(z)=\Ordo(|z-w_0|^n)$ near $w_0$. If $f$
is nontrivial, there exists an integer $N\ge n$ such that $f(z)=c\,(z-w_0)^N
+\Ordo(|z-w_0|^{N+1})$ near $w_0$, where $c\ne0$ is complex. At the same time,
the existence of such nontrivial $f$ entails that the corresponding root 
functions $\kernel_{m,N,w_0}$ is nontrivial as well, and  
that $K_{m,N,w_0}(\zeta,\zeta)\asymp|\zeta-w_0|^{2N}$ for $\zeta$ near $w_0$. 
On the other hand, the orthogonality between $f$ and $\kernel_{m,N,w_0}$ 
gives us that
\begin{multline*}
0=\langle f,\kernel_{m,N,w_0}\rangle_{mQ}=
\lim_{\zeta\to w_0}K_{m,N,w_0}(\zeta,\zeta)^{-\frac12}\big\langle f,
K_{m,N,w_0}(\cdot,\zeta)\big\rangle_{mQ}
\\
=\lim_{\zeta\to w_0}K_{m,N,w_0}(\zeta,\zeta)^{-\frac12} f(\zeta)
\end{multline*}
where we approach $w_0$ only in an appropriate direction so that the limit
exists. But this contradicts the given asymptotic behavior of $f(\zeta)$ near
$w_0$, since $c\ne0$ tells us that any limit of the right-hand side would
be nonzero. 
\end{proof}

\subsection{Off-spectral asymptotics of partial Bergman 
kernels}
Given a point $w_0\in\C$ we recall the partial Bergman spaces 
$A^2_{mQ,n,w_0}$, and the associated spectral droplets $\calS_{\tau,w_0}$ (see
\eqref{eq:calStau}), where we keep $n=\tau m$. 
Before we proceed with the formulation of the second result, let us fix some 
terminology. 

\begin{defn}\label{def:adm}
A real-valued potential $Q$ is said to be {\em $(\tau,w_0)$-admissible} if the 
following conditions hold:
\begin{enumerate}[(i)]
\item $Q:\C\to\R$ is $C^2$-smooth and has sufficient growth at infinity:
\[
\tau_Q:=\liminf_{\lvert z\rvert\to+\infty}\frac{Q(z)}{\log \lvert z\rvert}>0.
\]
\item $Q$ is real-analytically smooth and strictly subharmonic in a 
neighborhood of $\partial\calS_{\tau,w_0}$.
\item The point $w_0$ is an off-spectral point, i.e., 
$w_0\notin\calS_{\tau,w_0}$, and the component $\Omega_{\tau,w_0}$ of the 
complement $\calS_{\tau,w_0}^c$ containing the point $w_0$ is bounded and 
simply connected, with real-analytically smooth Jordan curve boundary.
\end{enumerate}
If for an interval $I\subset[0,+\infty[$, the potential $Q$ is 
$(\tau,w_0)$-admissible for each $\tau\in I$ and 
$\{\Omega_{\tau,w_0}\}_{\tau\in I}$ is a smooth flow of domains, then $Q$ is 
said to be $(I,w_0)$-admissible.
\end{defn}

Generally speaking, off-spectral components may be unbounded.
It is for reasons of simplicity that we focus on bounded off-spectral 
components in the above definition.

\medskip

\noindent {\sc A word on notation.} 
We will assume in the sequel that $Q$ is $(I,w_0)$-admissible for some 
non-trivial compact interval $I=I_{0}$. For an illustration of the situation,
see Figure~\ref{fig:mc}.
\begin{figure}
\centering
  \includegraphics[width=.6\linewidth,trim=0 70 0 0,clip]
{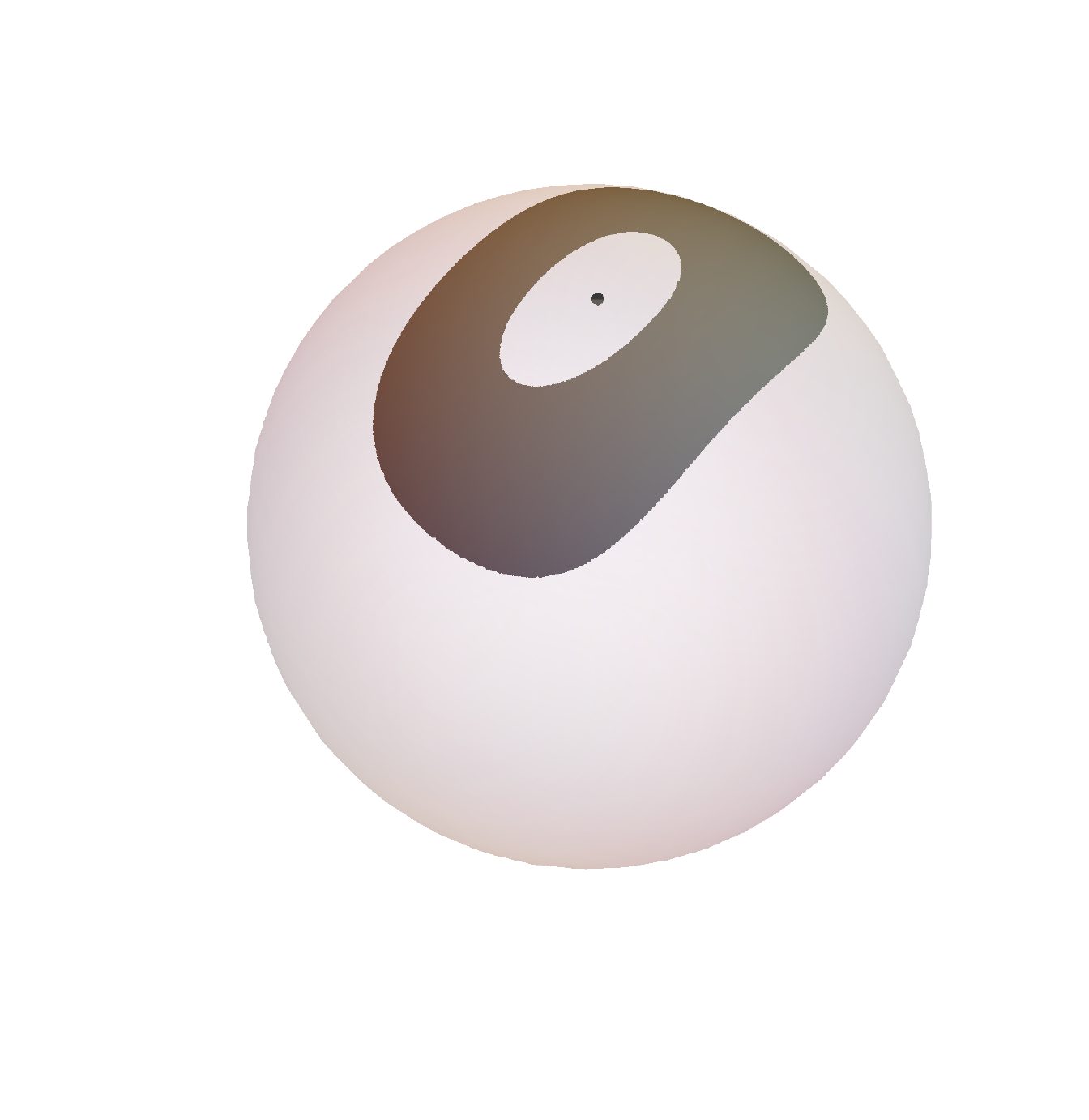}
\caption{
Illustration of a spectral droplet (shaded) in the Riemann sphere with two 
off-spectral components of the complement.
The point $w_0$ is indicated by the dot near 
the north pole, and the connectivity component of the off-spectral set 
containing $w_0$ is denoted by $\Omega_{w_0}$.}
\label{fig:mc}
\end{figure}
We let $\varphi_{\tau,w_0}$ denote the surjective Riemann mapping
\begin{equation}\label{eq:conf-map-bdd}
\varphi_{\tau,w_0}\colon \Omega_{\tau,w_0}\to\D,\quad 
\varphi_{\tau,w_0}(0)=0,\quad \varphi_{\tau,w_0}'(0)>0,
\end{equation}
which by our smoothness assumption on the boundary $\partial\Omega_{\tau,w_0}$ 
extends conformally across $\partial\Omega_{\tau,w_0}$.
We denote by $\mathcal{Q}_{\tau,w_0}$ the bounded holomorphic function in 
$\Omega_{\tau,w_0}$ whose real part equals $Q$ on $\partial\Omega_{\tau,w_0}$
and is real-valued at $w_0$. It is tacitly assumed to extend holomorphically
across the boundary $\partial\Omega_{\tau,w_0}$.   
We now turn to our second main result.

\begin{thm}
\label{thm:main1} 
Assume that the potential $Q$ is 
$(I_0,w_0)$-admissible, where the interval $I_0$ is compact.
Given a positive integer $\kappa$ and a positive real $A$, there exists
a neighborhood $\Omega^{\circledast}_{\tau,w_0}$ of the closure
of $\Omega_{\tau,w_0}$ and bounded holomorphic functions $\calB_{j,\tau,w_0}$
on $\Omega^{\circledast}_{\tau,w_0}$, as well as 
domains $\Omega_{\tau,w_0,m}=\Omega_{\tau,w_0,m,A}$ with 
$\Omega_{\tau,w_0}\subset\Omega_{\tau,w_0,m}\subset
\Omega^{\circledast}_{\tau,w_0}$ which meet
\[
\mathrm{dist}_{\C}(\Omega_{\tau,w_0,m}^c,
\Omega_{\tau,w_0})\ge Am^{-\frac12}(\log m)^{\frac12},
\]
such that the root function of order $n$ at $w_0$ enjoys the 
expansion
\[
\kernel_{m,n,w_0}(z)=
m^{\frac{1}{4}}(\varphi_{\tau,w_0}'(z))^{\frac12} (\varphi_{\tau,w_0}(z))^n
\e^{m\calQ_{\tau,w_0}(z)}\bigg\{\sum_{j=0}^{\kappa}m^{-j}
\calB_{j,\tau,w_0}(z)+\Ordo(m^{-\kappa-1})\bigg\}
\]
on $\Omega_{\tau,w_0,m}$ as $n=\tau m\to+\infty$ while $\tau\in I_0$,
where the error term is uniform.
Here, the main term $\calB_{0,\tau,w_0}$ is zero-free and smooth up to the 
boundary on $\Omega_{\tau,w_0}$ with $\calB_{0,\tau,w_0}(w_0)>0$, 
and with prescribed boundary modulus
\[
\lvert \calB_{0,\tau,w_0}(\zeta)\rvert=\pi^{-\frac14}[\hDelta Q(\zeta)]^{\frac14},
\qquad\zeta\in\partial\Omega_{\tau,w_0}.
\]
Moreover, for $A$ big enough, we have
\[
\int_{\Omega_{\tau,w_0,m}}|\kernel_{m,n,w_0}(z)|^2\e^{-2mQ(z)}\diffA(z)=
1+\Ordo(m^{-\kappa})
\]
as $n=\tau m\to+\infty$.
\end{thm}

\begin{rem} As in the case of the normalized Bergman kernels, the
expressions $\calB_{j,\tau, w_0}$ may be obtained algorithmically, for 
$j=1,2,3,\ldots$ (see Theorem~\ref{thm:comput-coeff} below). 
\end{rem}

\noindent {\sc Commentary to the theorem.}
{\sc(a)} As in Theorem~\ref{thm:main-prel}, the point $w_0$ may be 
chosen as any fixed point in $\Omega$, while 
the point $z$ is allowed to vary anywhere inside the set 
$\Omega_{\tau,w_0,m}$. Arguing in a fashion 
analogous to what we did in the commentary following 
Theorem~\ref{thm:main-prel}, we may conclude that 
$\kernel_{m,n,w_0}$ is small, both pointwise and in the $L^2$-sense, 
away from the set $\Omega_{\tau,w_0,m}$.

\noindent {\sc(b)} Our Theorems~\ref{thm:main-prel} and \ref{thm:main1} 
seemingly cover different instances of the root function asymptotics. 
In fact, in a certain sense, they are equivalent. Indeed, modulo 
technicalities we may obtain
Theorem~\ref{thm:main1} from Theorem~\ref{thm:main-prel} by replacing 
the potential 
$Q(z)$ with $\tilde{Q}_{\tau,w_0}(z)=Q(z)-\tau\log|z-w_0|$.
On the other hand, the former theorem is the limit case $\tau\to0$ of 
the latter.

\noindent {\sc(c)} Theorem \ref{thm:main1} should be compared with 
Theorem~4.1 in the work 
of Shiffman and Zelditch \cite{Shiffman-Zelditch}.
There, an asymptotic expansion of the diagonal restriction $K_{NP}(z,z)$ 
of the Bergman
kernel associated to a family of scaled Newton polytopes $NP$ is obtained, 
as $N\to+\infty$. This expansion
is valid deep inside the corresponding forbidden region.
In one complex dimension, this is equivalent to studying a certain 
weighted partial Bergman density 
for a polynomial Bergman space. Hence there is some overlap with the 
present work as well as with \cite{HW}.

\noindent {\sc(d)} The genuinely off-diagonal and off-spectral 
asymptotic expansion of the 
coherent states obtained here seem not to have any analogue elsewhere. 
Although we restrict ourselves to the setting of one complex variable,
we compensate for this by making minimal geometric assumptions.

\subsection{Interface transition of the density of states}
As a consequence of Theorem~\ref{thm:main1}, we obtain the 
transition behavior of the (partial) Bergman densities at emergent interfaces.
To explain how this works, we fix a bounded simply connected off-spectral 
component $\Omega$ with real-analytic Jordan curve boundary, 
associated to either a sequence of partial Bergman kernels $K_{m,n,w_0}$ of the 
space $A^2_{mQ,n,w_0}$ with the ratio $\tau=\frac{n}{m}$ fixed, 
or, alternatively, a sequence $K_m$ of full Bergman kernels. We think of
the latter as the parameter choice $\tau=0$.
We assume that the potential $Q$ is real-analytically smooth and strictly 
subharmonic near $\partial\Omega$.
Let $z_0\in\partial\Omega$, and denote by $\nu\in\T$ the inward unit normal to 
$\partial\Omega$ at $z_0$. We define the rescaled density 
$\varrho_m=\varrho_{m,\tau,w_0,z_0}$ by
\begin{equation}\label{eq:resc-density}
\varrho_m(\xi)=\frac{1}{2m\hDelta Q(z_0)}
K_{m,\tau m,w_0}\big(z_m(\xi), z_m(\xi)\big)\e^{-2mQ(z_m(\xi))},
\end{equation}
where the rescaled variable is defined implicitly by
\[
z_m(\xi)=z_0+\nu\frac{\xi}{\sqrt{2m\hDelta Q(z_0)}}.
\]

\begin{cor}
\label{cor:interface}
The rescaled density $\varrho_m$ in \eqref{eq:resc-density} has the limit
\[
\lim_{m\to+\infty}\varrho_m(\xi)=\mathrm{erf}\,(2\Re \xi)
=\frac{1}{\sqrt{2\pi}}\int_{2\Re \xi}^{\infty}\e^{-t^2/2}\diff t,
\]
where the convergence is uniform on compact subsets.
\end{cor}

\subsection{Comments on the exposition and a guide to the 
proofs}
In Section~\ref{s:off-spectral}, we explain the proofs of the main results, 
which are Theorems \ref{thm:main-prel} 
and \ref{thm:main1} together with Corollary~\ref{cor:interface}. 
Our approach is analogous to that of \cite{HW},
and we make an effort to explain exactly what needs to be modified for 
the techniques to apply in the present context.
As in \cite{HW}, the proofs involve the construction of a flow of loops 
near the spectral boundary, 
called the {\em orthogonal foliation flow}. 
The flow is constructed in an iterative procedure, which produces both 
the terms in the asymptotic expansion of the coherent state as well as
the successive terms in the expansion of the flow. 
The flow $\{\gamma_t\}_t$ is constructed to have the following property. 
If the normal velocity is denoted by $\nu$, we want the Szeg{\H{o}} kernel 
for the point $w_0$ with respect to the interior of the curve $\gamma_t$ 
and the induced weighted arc-length measure
$\e^{-2mQ}\nu\diff\sigma$ on $\gamma_t$ to be stationary in $t$. In turn, 
this allows us to glue together the Szeg\H{o} kernels to form the Bergman 
kernel.
In order to obtain the coefficient functions of the expansions in a more 
straightforward fashion, 
we apply the method developed in \cite{HW} which is based on Laplace's method. 
Another important ingredient, which allows localization to a neighborhood 
of each off-spectral component, is H{\"o}rmander's $\bar\partial$-estimate, 
suitably modified to the given needs.
Since this is the method which allows us to change the geometry drastically, 
we sometimes refer to it as {\em $\bar\partial$-surgery}.

In Section~\ref{s:off-spectral-general}, we develop a more general version 
of the foliation flow lemma, which allows us to introduce a conformal 
factor in the area form. 
To avoid unnecessary repetition, both the orthogonal foliation flow and 
the algorithm for computing the coefficient functions in the asymptotic 
expansions are explained only in this more general setting.
We supply a couple of applications of this extension, 
including a stability result for the root functions and
the orthogonal polynomials under a $\frac1m$-perturbation of the potential $Q$ 
(Theorems~\ref{thm:twist-root} and \ref{thm:twist-onp}). 

\subsection{Acknowledgements} 
We wish to thank Steve Zelditch, Bo Berndtsson, and Robert Berman for their 
interest in this work. In addition we should thank the anonymous referee for 
his or her helpful efforts. 

\section{Off-spectral expansions of normalized kernels} 
\label{s:off-spectral}
\subsection{A family of obstacle problems and evolution 
of the spectrum}
\label{ss:obst}
The spectral droplets \eqref{eq:calS} and the partial analogues 
\eqref{eq:calStau} were defined earlier. From that point of view, 
the spectral droplet $\calS$ is the instance $\tau=0$ of the partial spectral
droplets $\calS_{\tau,w_0}$. We should like to point out here that the partial
spectral droplet $\calS_{\tau,w_0}$ emerges as the full spectrum under a 
perturbation of the potential $Q$. To see this, we consider the perturbed
potential
\[
\tilde Q(z)=\tilde Q_{\tau,w_0}(z):=Q(z)-\tau\log|z-w_0|,
\]
and observe that the coincidence set $\tilde\calS$ for $\tilde Q$ equals 
the partial spectral droplet  $\calS_{\tau,w_0}$. 

The following proposition summarizes some basic properties of the function
$\hat{Q}_{\tau,w_0}$ given by \eqref{eq:obst-function-tau}. 
We refer to \cite{HM} for the necessary details.

\begin{prop}\label{prop:obst-sol}
Assume that $Q\in C^2(\C)$ is real-valued with the logarithmic growth of
condition {(i)} of Definition \ref{def:adm}. 
Then for each $\tau$ with $0\le\tau<\tau_Q$ and for each point $w_0\in\C$, 
the function $\hat{Q}_{\tau,w_0}$ is a 
subharmonic function in the plane $\C$ which is $C^{1,1}$-smooth off $w_0$, 
and harmonic on $\C\setminus(\calS_{\tau,w_0}\cup\{w_0\})$. 
Near the point $w_0$ we have 
\[
\hat{Q}_{\tau,w_0}(z)=\tau\log \lvert z-w_0\rvert+\Ordo(1).
\]
\end{prop}

The evolution of the free boundaries $\partial\calS_{\tau,w_0}$, which is of 
fundamental importance for our understanding of the properties of the 
normalized reproducing kernels, is summarized in the following.

\begin{prop}
\label{prop:free-bdry}
The continuous chain of off-spectral components 
$\Omega_{\tau,w_0}$ for $\tau\in I_0$ deform according to weighted 
Laplacian growth with weight $2\hDelta Q$, that is, for 
$\tau,\tau'\in I_0$ with $\tau'<\tau$, and for any bounded harmonic 
function $h$ on $\Omega_{\tau,w_0}$, we have that
\[
\int_{\Omega_{\tau,w_0}\setminus\Omega_{\tau',w_0}}h\,2\hDelta Q\diffA
=(\tau-\tau')h(w_0).
\]
Fix a point $\zeta\in\partial\Omega_{\tau,w_0}$, and denote for real 
$\varepsilon$ by $\zeta_{\varepsilon,w_0}$ the point closest to $\zeta$ in the 
intersection
\[
(\zeta+\nu_\tau(\zeta)\R_+)\cap \partial\Omega_{\tau-\varepsilon,w_0},
\]
where $\nu_\tau(\zeta)\in\T$ points in the inward normal direction at $\zeta$
with respect to $\Omega_{\tau,w_0}$.
Then we have that
\[
\zeta_{\varepsilon}=\zeta+\varepsilon\, \nu_\tau(\zeta)
\frac{\lvert \varphi'_{\tau,w_0}(\zeta)\rvert}{4\hDelta Q(\zeta)}
+\Ordo(\varepsilon^2),\quad \varepsilon\to0,
\]
and the outer normal ${\rm n}_{\tau-\varepsilon,w_0}(\zeta_\varepsilon)$ satisfies
\[
{\rm n}_{\tau-\varepsilon,w_0}(\zeta_\varepsilon)={\rm n}_{\tau,w_0}(\zeta)
+\Ordo(\varepsilon).
\]
\end{prop}

\begin{proof}
That the domains deform according to Hele-Shaw flow is a direct 
consequence of the
relation of $\Omega_{\tau,w_0}$ to the obstacle problem. 
To see how it follows, assume that $h$ is harmonic on $\Omega_{\tau,w_0}$ and
$C^2$-smooth up to the boundary, and apply Green's formula to obtain
\begin{multline*}
\int_{\Omega_{\tau,w_0}}h(z)\hDelta Q(z)\diffA(z)=
\frac{1}{4\pi}\int_{\partial\Omega_{\tau,w_0}}
\Big(h(z)\partial_{\rm n}Q(z)-Q(z)\partial_{\rm n}h(z)\Big)\lvert \diff z\rvert 
\\
= \frac{1}{4\pi}\int_{\partial\Omega_{\tau,w_0}}
\Big(h(z)\partial_{\rm n}\hat{Q}_{\tau,w_0}(z)-\hat{Q}_{\tau,w_0}(z)
\partial_{\rm n}h(z)\Big)
\lvert \diff z\rvert 
\\
=\int_{\Omega_{\tau,w_0}}h(z) \hDelta\hat{Q}_{\tau,w_0}\diffA(z),
\end{multline*}
where the latter integral is understood in the sense of distribution theory.
As $\hat{Q}_{\tau,w_0}$ is a harmonic perturbation of $\tau$ times the 
Green function for $\Omega_{\tau,w_0}$, the result follows by writing
$\int_{\Omega_{\tau,w_0}\setminus\Omega_{\tau',w_0}}h\hDelta Q\diffA$ as the 
difference of two integrals of the above form, and by approximation of bounded
harmonic functions by harmonic functions $C^2$-smooth up to the boundary.

The second part follows along the lines of \cite[Lemma~2.3.1]{HW}.
\end{proof}

We turn next to an off-spectral growth bound for weighted
holomorphic functions.

\begin{prop}
\label{prop:growth}
Assume that $Q$ is admissible and denote by $\calK_{\tau,w_0}$ a closed subset 
of the interior of $\calS_{\tau,w_0}$.
Then there exist constants $c_0$ and $C_0$ such that for any 
$f\in A^2_{mQ,n,w_0}(\C)$ it holds that
\[
\lvert f(z)\rvert\le C_0\,m^{\frac12}\e^{m\hat{Q}_{\tau,w_0}}
\big\lVert 1_{\calK_{\tau,w_0}^{\mathrm{c}}}\,f\big\rVert_{mQ},\qquad 
\mathrm{dist}(z,\calK_{\tau,w_0})\ge c_0m^{-\frac12}.
\]
In case $\calK_{\tau,w_0}=\emptyset$, the estimate holds globally.
\end{prop}
\begin{proof}
This follows immediately by an application of the maximum principle,
together with the result of Lemma~2.2.1 in \cite{HW}, originating 
from \cite{ahm1}.
\end{proof}

\subsection{Some auxilliary functions}
There are a number of functions related to the potential $Q$ that will be 
useful in the sequel.
We denote by $\calQ_{\tau,w_0}$ the bounded 
holomorphic function on $\Omega_{\tau,w_0}$ whose real part
on the boundary curve $\partial\Omega_{\tau,w_0}$ equals $Q$,
uniquely determined by the requirement that $\Im \calQ_{\tau,w_0}(w_0)=0$. 
We also need the function $\breve{Q}_{\tau,w_0}$, which denotes the harmonic 
extension of $\hat{Q}_{\tau,w_0}$ across the boundary of the off-spectral 
component $\Omega_{\tau,w_0}$.
These two functions are connected via
\begin{equation}
\breve{Q}_{\tau,w_0}(z)=\tau\log\lvert \varphi_{\tau,w_0}(z)\rvert
+\Re\calQ_{\tau,w_0}(z).
\label{eq:Qbreve101}
\end{equation} 
Since we work with $(\tau,w_0)$-admissible potentials $Q$,  the off-spectral 
component $\Omega_{\tau,w_0}$ is a bounded simply 
connected domain with real-analytically smooth Jordan curve boundary. 
Without loss of generality, we may hence assume that
$\calQ_{\tau,w_0}$, $\breve{Q}_{\tau,w_0}$  as well as the conformal mapping 
$\varphi_{\tau,w_0}$ extend to a common domain $\Omega_0$, containing the closure
$\bar{\Omega}_{\tau,w_0}$. By possibly shrinking the interval $I_0$, we 
may moreover choose the set $\Omega_0$ to be independent of the parameter 
$\tau\in I_0$.

\subsection{Canonical positioning}
An elementary but important observation for the main result of 
\cite{HW} is that we may ignore a compact subset of the interior of the 
compact spectral droplet $\calS_\tau$ associated to polynomial Bergman kernels 
when we study the asymptotic expansions of the orthogonal polynomials 
$P_{m,n}$ (with $\tau=\frac{n}{m}$). Indeed, only the behavior in a small 
neighborhood of the complement $\calS_\tau^{\mathrm{c}}$ is of 
interest, and the $\bar\partial$-surgery methods allow us to disregard the 
rest. 
The physical intuition behind this is the interpretation of the
probability density $|P_{m,n}|^2\e^{-2mQ}$ as the net effect of adding one more
particle to the system, and since the positions in the interior of the droplet 
are already occupied we would expect the net effect to occur near
the boundary.  
The fact that we may restrict our attention to a simply connected 
proper subset of the Riemann sphere $\hat\C$ breaks up the rigidity and allows 
us to apply a conformal mapping to place ourselves in an appropriate
model situation. 

In the present context, we consider the Riemann mapping $\varphi_{\tau,w_0}$
which maps the off-spectral region $\Omega_{\tau,w_0}$ onto the unit disk 
$\D$, and 
has a conformal extension to a neighborhood of $\overline{\Omega}_{\tau,w_0}$.
We let the {\em canonical positioning operator} for the point $w_0$ 
with respect to
the off-spectral component $\Omega_{\tau,w_0}$ be given by
\begin{equation}
\label{eq:c-pos}
\Vop_{m,n,w_0}[v]=\varphi_{\tau,w_0}'(z)(\varphi_{\tau,w_0})^n
\e^{m\calQ_{\tau,w_0}}v\circ\varphi_{\tau,w_0},
\end{equation}
and put 
\begin{equation}
R_{\tau,w_0}:=(Q-\breve{Q}_{\tau})\circ\varphi_{\tau,w_0}^{-1}.
\label{eq-Rtau}
\end{equation}
An essential property of this operator is that $\Vop_{m,n,w_0}$ acts 
isometrically
from $L^2(\e^{-2mR_{\tau,w_0}})$ to $L^2(\e^{-2mQ})$, wherever it is well-defined.

As for the coherent states, we will analyze them in terms of the canonical 
positioning operator $\Vop_{m,n,w_0}$.
This simplifies the geometry of $\Omega_{\tau,w_0}$ by mapping it to the 
unit disk $\D$, and simplifies the weight.
Indeed, the function $R_{\tau,w_0}$ is flat to order $2$ at the unit circle, 
and consequently the weight 
$\e^{-2mR_{\tau,w_0}}$ behaves like a Gaussian ridge. 

We summarize the properties of the operator $\Vop_{m,n,w_0}$ in the 
following proposition.
For a potential $V$ and a domain $\Omega$ with $w_0\in\Omega$, we denote by 
$A^2_{mV,n,w_0}(\Omega)$ the space of 
holomorphic functions on $\Omega$ which vanish to order $n$ at $w_0\in\Omega$, 
endowed with the topology of $L^2(\mathrm{e}^{-2mV},\Omega)$. 
In case $n=0$ we simply denote the space by $A^2_{mV}(\Omega)$.

\begin{prop}\label{prop:Vop} 
Let $Q$ be a $(\tau,w_0)$-admissible potential, and let $\Omega_{\tau,w_0}$
denote the corresponding off-spectral component. Moreover, let 
$R_{\tau, w_0}$ be given by \eqref{eq-Rtau}. Then, for $\hdelta>1$ 
sufficiently close to $1$, the operator $\Vop_{m,n,w_0}$ defines an 
invertible isometry
\[
\Vop_{m,n,w_0}:A^2_{mR_{\tau,w_0}}\big(\D(0,\hdelta)\big)\to 
A^2_{mQ,n,w_0}(\Omega_{0}),
\]
if $\Omega_0=\varphi_{\tau,w_0}^{-1}(\D(0,\hdelta))$.  
The isometry property 
remains valid in the context of weighted $L^2$-spaces as well.
\end{prop}

\begin{proof}
The conclusion is immediate by the defining normalizations of the conformal 
mapping $\varphi_{\tau,w_0}$.
\end{proof}

The following definition is an analogue of Definition~3.1.2 in \cite{HW}. 
We denote by $\Omega_1$ a domain containing the closure of the off-spectral 
component $\Omega_{\tau,w_0}$, and let $\chi_{0,\tau}$ denote a $C^\infty$-smooth 
cut-off function which vanishes off $\Omega_1$, and equals $1$ in a 
neighborhood of the closure of $\Omega_{\tau,w_0}$.

\begin{defn}\label{def:q-pol-orth}
Let $\kappa$ be a positive integer.
A sequence $\{F_{m,n,w_0}\}_{m,n}$ of holomorphic functions on $\Omega_0$ 
is called a \emph{sequence of approximate root functions of order $n$ at $w_0$ 
of accuracy} $\kappa$ for the space $A^2_{mQ,n,w_0}$
if the following conditions are met as $m\to+\infty$ 
while $\tau=\frac{n}{m}\in {I}_{w_0}$:

\noindent (i) For all $f\in A^2_{mQ,n+1,w_0}$, we have the approximate 
orthogonality
\[
\int_{\C}\chi_{0,\tau}F_{m,n,w_0}\bar{f}\,\e^{-2mQ}\diffA=
\Ordo(m^{-\kappa-\frac13}\lVert p\rVert_{mQ}).
\]

\noindent (ii) The approximate root functions have norm approximately equal 
to $1$,
\[
\int_{\C}\chi_{0,\tau}^2\lvert F_{m,n,w_0}(z)\rvert^2\e^{-2mQ}(z)\diffA(z)=
1+\Ordo(m^{-\kappa-\frac13}).
\]

\noindent (iii) The functions $F_{m,n,w_0}$ are approximately real and positive
at $w_0$, in the sense that the leading coefficient 
$a_{m,n,w_0}=\lim_{z\to w_0} (z-w_0)^{-n}F_{m,n,w_0}(z)$ satisfies 
\break{$\Re a_{m,n,w_0}>0$} and
\[
\frac{\Im a_{m,n,w_0}}{\Re a_{m,n,w_0}}=\Ordo(m^{-\kappa-\frac{1}{12}}).
\]
\end{defn}

We remark that the exponents in the above error terms are chosen for 
reasons of convenience, related to the correction scheme of 
Subsection~\ref{ss:d-bar}.

\subsection{The orthogonal foliation flow}
The orthogonal foliation flow $\{\gamma_{m,n,t}\}_t$ is a smooth flow 
of closed curves near the unit circle $\T$, originally formulated in 
\cite{HW} in the context of orthogonal polynomials.
The defining property is that $P_{m,n}$ should be approximately orthogonal 
to the lower degree polynomials along the curves 
$\Gamma_{m,n,t}=\phi_\tau^{-1}(\gamma_{m,n,t})$ with respect 
to the induced measure $\e^{-2mQ}\nu_{\mathrm{n}}\diffs$, where 
$\nu_{\mathrm{n}}$ denotes the normal velocity of the flow $\{\Gamma_{m,n,t}\}_t$
and $\diffs$ denotes normalized arc length measure.

\medskip

\noindent{\sc Smoothness classes and polarization of smooth functions.} 
We fix the smoothness class of the weights under consideration, and adapt
Definition~4.2.1 in \cite{HW} to the the present setting. 
First, we need the notion of polarization, which applies to real-analytically 
smooth functions. If $R(z)$ is real-analytic, there exists a function of two 
complex variables, denoted by $R(z,w)$, which is holomorphic in $(z,\bar w)$
in a neighborhood of the diagonal, with diagonal restriction $R(z,z)=R(z)$. 
The function $R(z,w)$ is referred to as the \emph{polarization} of $R(z)$, 
and it is uniquely determined by its diagonal restriction $R(z)$. 
If $R(z,w)$ is such a polarization of a function $R(z)$ which is 
real-analytically smooth near the circle $\T$ and quadratically flat there, 
then $R(z)=(1-|z|^2)^2R_0(z)$ and in polarized form 
$R(z,w)=(1-z\bar w)^2R_0(z,w)$, where $R_0(z,w)$ is holomorphic in 
$(z,\bar w)$ in a neighborhood of the diagonal where both variables are near 
$\T$. The function $R_0(z,w)$ is then the polarization of $R_0(z)$.

\begin{defn}
\label{def:class-W}
For real numbers $\hdelta,\sigma$ with $\hdelta>1$ and $\sigma>0$, 
we denote by $\mathfrak{W}(\hdelta,\sigma)$ the class of non-negative  
$C^2$-smooth functions $R$ on $\D(0,\hdelta)$ such that $R$ is 
quadratically flat on $\T$ with $\hDelta R|_\T>0$ and satisfies
\[
\inf_{z\in \D(0,\hdelta)}
R_0(z)=\frac{R(z)}{(1-|z|^2)^2}=\alpha(R)>0,
\]
while on the annulus 
\[
\mathbb{A}(\hdelta^{-1},\hdelta):=\{z\in\C:\,\hdelta^{-1}<|z|<\hdelta\}
\]
$R$ is real-analytically smooth and has a polarization $R(z,w)$ which is 
holomorphic in $(z,\bar w)$ on the $2\sigma$-fattened diagonal annulus
\[
\hat{\mathbb{A}}(\hdelta,\sigma)
=\big\{(z,w)\in\mathbb{A}(\hdelta^{-1},\hdelta)\times 
\mathbb{A}(\hdelta^{-1},\hdelta)\;:\;\lvert z-w\rvert\le 2\sigma\big\},
\]
and factors as $R(z,w)=(1-z\bar w)^2R_0(z,w)$, where $R_0(z,w)$ is
holomorphic $(z,\bar w)$ on the set 
$\hat{\mathbb{A}}(\hdelta,\sigma)$, and bounded 
and bounded away from zero there.
We say that a subset $S\subset \mathfrak{W}(\hdelta,\sigma)$ is a 
\emph{uniform family}, provided that for each $R\in S$, the corresponding
$R_0(z,w)$ is uniformly bounded and bounded away from $0$ on 
$\hat{\mathbb{A}}(\hdelta,\sigma)$
while the constant $\alpha(R)$ is uniformly bounded away from $0$.
\end{defn}

The point with above definition is that it lets us encode 
uniformity properties of the potentials $R_{\tau,w_0}$ with respect to the 
parameter $\tau$ and the point $w_0$.

For a polarized function $f(z,w)$, we let $f_\T(z)$ denote the restriction of 
$f(z,z)$ for $z\in\T$, wherever it is well-defined. We recall from 
Proposition~4.2.2 of \cite{HW} that if $f(z,{w})$ is holomorphic in 
$(z,\bar w)$ on the $2\sigma$-fattened diagonal annulus 
$\hat{\mathbb{A}}(\hdelta,\sigma)$ and if the parameters meet 
$1<\hdelta\le \sqrt{1+\sigma^2}+\sigma$, then it follows that the function 
$f_{\T}$ extends holomorphically to the annulus 
$\mathbb{A}(\hdelta^{-1},\hdelta)$. We may need to restrict the
numbers $\hdelta,\sigma$ further. Indeed, it turns out that we need 
that the functions
$\log\hDelta R$, $\hat{R}=\sqrt{R}$ 
(chosen to be positive inside the unit circle and negative outside) 
as well as $\log(-z\partial_z\hat{R})$ have polarizations which are 
holomorphic in $(z,\bar{w})$ for $(z,w)\in\hat{\mathbb{A}}(\hdelta,\sigma)$ and 
uniformly bounded there as well. 
If $R$ belongs to a uniform family of $\mathfrak{W}(\hdelta_0,\sigma_0)$, 
then there exist 
$(\hdelta_1,\sigma_1)$ such that these properties
hold for the polarizations with 
$\hdelta=\hdelta_1$ and $\sigma=\sigma_1$ (See Proposition~4.2.3 of \cite{HW}),
where we moreover require that $1<\hdelta_1\le \sqrt{1+\sigma_1^2}+\sigma_1$.

\begin{lem}
\label{lem:R-weight} 
Let $\calK$ be a compact subset of each of 
the domains $\Omega_{\tau,w_0}$, where $\tau\in I_0$.
Then there exist constants $\hdelta,\sigma$ with $\hdelta>1$ and $\sigma>0$,
 such that the 
collection of weights $R_{\tau,w_0}$ with $w_0\subset\calK$ and $\tau\in I_0$
is a uniform family in $\mathfrak{W}(\hdelta,\sigma)$.
\end{lem}

This is completely analogous to the corresponding claim in of \cite{HW}, 
which was expressed in the context of an exterior conformal mapping.

\medskip

\noindent{\sc The orthogonal foliation flow near the unit circle.} 
The existence of the orthogonal foliation flow around the circle $\T$ and 
the asymptotic expansion of the root functions after canonical positioning 
are stated in the following lemma (compare with Lemma~4.1.2 in \cite{HW}).
For the proof, we refer to the sketched proof of 
Lemma~\ref{lem:main-flow-general}
below, as well as the complete proof of Lemma~4.1.2 in \cite{HW}, for 
the case of orthogonal
polynomials. 

\begin{lem}\label{lem:main-flow}
Fix an accuracy parameter $\kappa$ and let 
$R\in \mathfrak{W}(\hdelta_0,\sigma_0)$.
Then, if $\hdelta_1$ is as above,
there exist a radius
$\hdelta_2$ with $1<\hdelta_2<\hdelta_1$, bounded holomorphic functions 
$f_{s}$ on $\D(0,\hdelta_1)$ of the form
\[
f_{s}(z)=\sum_{0\le j\le\kappa}s^{j}
B_{j}(z),\qquad z\in\D(0,\hdelta'),
\]
and conformal mappings $\psi_{s,t}$ from $\D(0,\rho_2)$ into the plane
given by 
\[
\psi_{s,t}=\psi_{0,t}+\sum_{\substack{(j,l)\in\indsett_{2\kappa+1}\\ j\ge 1}}s^j t^l 
\hat{\psi}_{j,l}
\] 
such that for $s,t$ small enough, the domains 
$\psi_{s,t}\big(\D\big)$ grow with $t$, while they remain 
contained in $\D(0,\hdelta_1)$. Moreover, for 
$\zeta\in\T$ we have that
\begin{multline}
\label{eq:flow-eq}
\lvert f_{s}\circ \psi_{s,t}(\zeta)\rvert^2\,\e^{-2s^{-1}R\circ\psi_{s,t}}
\Re\big(\bar{\zeta}\partial_t\psi_{s,t}(\zeta)\overline{\psi_{s,t}'(\zeta)}\big)
\\
=\e^{-s^{-1}t^2}\Big\{(4\pi)^{-\frac12}+\Ordo\big(\lvert s\rvert^{\kappa+\frac12}
+\lvert t\rvert^{2\kappa+1}\big)\Big\}.
\end{multline}
For small positive $s$, when $t$ varies in the interval $[-\beta_s,\beta_s]$ 
with $\beta_s:=s^{1/2}\log \frac{1}{s}$, the flow of loops 
$\{\psi_{s,t}(\T)\}_t$ cover a neighborhood of the circle $\T$ of width 
proportional to $\beta_s$ smoothly.
In addition, the first term $B_0$ is zero-free, positive at the origin, 
and has modulus $|B_0|=\pi^{-\frac14}(\hDelta R)^{\frac14}$ on $\T$. 
The other terms $B_j$ are
all real-valued at the origin. 
The implied constant in \eqref{eq:flow-eq} is uniformly bounded, provided that 
$R$ is confined to a uniform family of $\mathfrak{W}(\hdelta_0,\sigma_0)$. 
\end{lem}

\subsection{$\bar\partial$-corrections and asymptotic 
expansions of root functions}
\label{ss:d-bar}
In this section, we supply a proof of the main result, Theorem~\ref{thm:main1}.
The proof consists of two parts. First, we construct a family of approximate 
root function of a given order and accuracy, 
after which we apply H{\"o}rmander-type $\bar\partial$-estimates to correct 
these approximate kernels to entire functions. The precise result needed 
for the correction scheme runs as follows.

\begin{prop}\label{prop:horm}
Let $f\in L^\infty(\calS_{\tau,w_0})$, where $\tau=\frac{n}{m}$, 
and denote by $u=u_{m,n,w_0}$ the norm-minimal solution in 
$L^2_{m\hat{Q}_{\tau,w_0}}$ to the problem
\[
\bar\partial u=f,
\]
among the functions which vanish at $w_0$ to order $n$:
$|u(z)|=\Ordo(|z-w_0|^n)$ around $w_0$. 
Then $u$ meets the bound
\[
\int_{\C}\lvert u\rvert^2\e^{-2m\hat{Q}_{\tau,w_0}}\diffA\le 
\frac{1}{2m}\int_{\calS_{\tau,w_0}}\lvert f\rvert^2
\frac{\e^{-2mQ}}{\hDelta Q}\diffA.
\]
\end{prop}

This is an immediate consequence of Corollary~2.4.2 in \cite{HW}, and 
essentially amounts to H\"ormander's classical bound for the 
$\bar\partial$-equation in the given setting.

We turn to the proof of Theorem~\ref{thm:main1}. 

\begin{proof}[Sketch of proof of Theorem~\ref{thm:main1}]
As the proof is analogous to that of Theorems 1.3.3 and 1.3.4 in
\cite{HW}, we supply only an outline of the proof.
\\
\\
\noindent {\sc The construction of approximate root functions.} We apply 
Lemma~\ref{lem:main-flow} with $s=m^{-1}$ and 
$R=R_{\tau,w_0}$ to obtain a smooth flow $\gamma_{s,t}=\gamma_{m,n,t,w_0}$ 
of curves, as well as bounded holomorphic functions 
$f_{s}=f_{m,n,w_0}^{\langle\kappa\rangle}$ such that the flow equation 
\eqref{eq:flow-eq} is met. 
The sequence $\{B_j\}_j$ of bounded holomorphic functions 
produced by the lemma actually depend (smoothly) on 
the parameter $\tau$ and the root point $w_0$, so we put
$B_j=B_{j,\tau,w_0}$ and define
\begin{equation}
\label{def:hnm}
f_{m,n,w_0}^{\langle \kappa\rangle}=\sum_{j=0}^{\kappa}m^{-j}B_{j,\tau,w_0}.
\end{equation}
If we write
\[
\calB_{j,\tau,w_0}:=(\varphi_{\tau,w_0}')^{\frac{1}{2}}
B_{j,\tau,w_0}\circ\varphi_{\tau,w_0},
\]
it follows that 
\[
\kernel_{m,n,w_0}^{\langle \kappa\rangle}:=m^{\frac14}\Vop_{m,n,w_0}
[f_{m,n,w_0}^{\langle \kappa\rangle}]
\] 
has the claimed form.
It remains to show that $\kernel_{m,n,w_0}^{\langle \kappa\rangle}$
is a family of approximate root functions of order $n$ at $w_0$ 
with the stated uniformity property,
and to show that it is close to the true normalizing reproducing kernel.

We denote by $\calD_{m,n,w_0}$ the domain covered by the foliation flow,
over the parameter range $-\delta_m\le t\le\delta_m$, where 
$\delta_m:=m^{-\frac12}\log m$.
Moreover, we define the domain $\calE_{\tau,w_0}$ as the image of 
$\D(0,\hdelta'')$ under $\varphi_{\tau,w_0}^{-1}$, and let 
$\chi_{0}=\chi_{0,\tau,w_0}$ denote an an appropriately chosen smooth 
cut-off function, which takes the value $1$ on a neighborhood of 
$\bar\Omega_{\tau,w_0}$
and vanishes off $\calE_{\tau,w_0}$. 
If we let $\chi_{1}:=\chi_{0}\circ\varphi_{\tau,w_0}^{-1}$ denote the 
corresponding cut-off extended to vanish off $\D(0,\hdelta'')$, 
we may show that
\begin{equation}
\label{eq:norm-h-flow}
\int_{\calD_{m,n,w_0}}
\lvert f_{m,n,w_0}^{\langle \kappa\rangle}\rvert^2\e^{-2mR_{\tau,w_0}}
 \diffA(z)=m^{-\frac12}+\Ordo(m^{-\kappa-\frac{5}{6}}),
\end{equation}
as follows immediately from integration of the flow equation of 
Lemma~\ref{lem:main-flow},
as well as the estimate
\begin{equation}
\int_{\C\setminus\calD_{m,n,w_0}}\chi_{1}^2 
\lvert f_{m,n,w_0}^{\langle \kappa\rangle}\rvert^2\e^{-2mR_{\tau,w_0}}
 \diffA(z)=\Ordo(m^{-\alpha_0\log m}),
\label{eq:outsideflow1}
\end{equation}
which holds for some $\alpha_0>0$ as a consequence of the Gaussian 
ridge behavior of the function 
$\e^{-2mR_{\tau,w_0}}$ around the unit circle. We now observe that in 
view of \eqref{eq:norm-h-flow} and 
\eqref{eq:outsideflow1}, the isometric property of $\Vop_{m,n,w_0}$ implies that 
$\chi_0\kernel_{m,n,w_0}^{\langle \kappa\rangle}$ has norm $1+\Ordo(m^{-\kappa-\frac13})$ 
in $L^2_{mQ}$.

Let $g\in A^2_{mQ,n,w_0}$ be given, and put
$q=m^{-\frac14}\Vop_{m,n,w_0}^{-1}[g]$.
Then, by the isometric property of $\Vop_{m,n,w_0}$ and the estimate 
\eqref{eq:outsideflow1}, it follows that
\begin{multline*}
\int_{\C}\chi_{0}\,\kernel_{m,n,w_0}^{\langle \kappa\rangle}(z) \bar{g}(z)\,\e^{-2mQ}
\diffA(z) 
=m^{\frac12}\int_{\C}\chi_{1} f_{m,n,w_0}^{\langle \kappa\rangle}(z)\bar{q}(z)\,
\e^{-2mR_{\tau,w_0}}\diffA(z)\\
=m^{\frac12}\int_{\calD_{m,n,w_0}}\chi_{1} 
f_{m,n,w_0}^{\langle \kappa\rangle}(z)
\bar{q}(z)\,\e^{-2mR_{\tau,w_0}}\diffA(z)
+\Ordo\big(m^{-\frac{\alpha_0}{2}\log m+\frac14}\lVert g\rVert_{mQ}\big),
\end{multline*}
where we have applied the Cauchy-Schwarz inequality together with 
the estimate \eqref{eq:outsideflow1} to obtain the error term.
 
The function $f_{m,n,w_0}^{\langle \kappa\rangle}$ is zero-free up to the boundary in 
$\calD_{m,n,w_0}$ provided that $m$ is large enough, as the main term 
is bounded away from $0$ in modulus, 
and consecutive terms are much smaller. Also, for large enough $m$, 
it holds that $\chi_1=1$ on $\calD_{m,n,w_0}$. We now introduce the function 
\[
q_{m,n}:=\frac{q}{f_{m,n}^{\langle \kappa\rangle}}
\]
and integrate along the flow:
\begin{multline}\label{eq:comput-fol-coord}
m^{\frac12}\int_{\calD_{m,n,w_0}}\chi_{1} f_{m,n,w_0}^{\langle \kappa\rangle}(z)
\bar{q}(z)\,\e^{-2mR_{\tau,w_0}}\diffA(z)\\
=m^{\frac12}\int_{\calD_{m,n,w_0}}q_{m,n}(z)\lvert 
f_{m,n,w_0}^{\langle\kappa\rangle}(z)\rvert^2 
\e^{-2mR_{\tau,w_0}(z)}\diffA(z)
\\
=2m^{\frac12}\int_{-\delta_m}^{\delta_m}\int_{\T}
q_{m,n}\circ\psi_{m,n,t}
(\zeta)\big\lvert f_{m,n,w_0}^{\langle \kappa\rangle}\circ\psi_{m,n,t}
(\zeta)
\big\rvert^2\e^{-2m\,R_{\tau,w_0}\circ\psi_{m,n,t}(\zeta)}
\\
\times
\Re\big\{\bar{\zeta}\partial_t\psi_{m,n,t}(\zeta)
\overline{\psi_{m,n,t}'(\zeta)}\big\}\diffs(\zeta)\diff t
\\=2m^{\frac12}\int_{-\delta_m}^{\delta_m}
\int_{\T}q_{m,n}\circ\psi_{m,n,t}(\zeta)\Big\{(4\pi)^{-\frac12}\e^{-mt^2}
+\Ordo\big(m^{-\kappa-\frac13}\e^{-mt^2}\big)\Big\}\diffs(\zeta)\diff t,
\end{multline}
where the last step uses the flow equation of Lemma~\ref{lem:main-flow}.
We now make the crucial observation is that for fixed $t$,
the composition $q_{m,n}\circ\psi_{m,n,t}$ is 
holomorphic, so that we may apply the mean value property:
\[
\int_{\T}q_{m,n}\circ \psi_{m,n,t}\diffs=q_{m,n}\circ \psi_{m,n,t}(0)=q_{m,n}(0)
=\frac{q(0)}{m^{\frac14}f_{m,n}^{\langle\kappa\rangle}(0)}.
\]
Consequently, it follows that
\begin{multline}
2m^{\frac12}\int_{-\delta_m}^{\delta_m}
\int_{\T}q_{m,n}\circ\psi_{m,n,t}(\zeta)\Big\{(4\pi)^{-\frac12}\e^{-mt^2}
+\Ordo\big(m^{-\kappa-\frac13}\e^{-mt^2}\big)\Big\}\diffs(\zeta)\diff t
\\
=q_{m,n}(0)
\big(1+\Ordo(m^{-\log m})\big)
\\
+\Ordo\bigg(
m^{-\kappa+\frac1{6}}\int_{-\delta_m}^{\delta_m}\int_{\T}\lvert q_{m,n}
\circ\psi_{m,n,t}(\zeta)\rvert \diffs(\zeta)\,\e^{-mt^2}\diff t\bigg)
\\
=q_{m,n}(0)
+\Ordo\big(m^{-\kappa-\frac1{3}}\|g\|_{mQ}\big)
=\frac{q(0)}{f_{m,n}^{\langle\kappa\rangle}(0)}
+\Ordo\big(m^{-\kappa-\frac1{3}}\|g\|_{mQ}\big).
\label{eq:flowcalc102}
\end{multline}
Here, in order to obtain the last estimate, we have used 
\eqref{eq:comput-fol-coord} backwards with $q_{m,n}$
replaced by $|q_{m,n}|$, and the estimate \eqref{eq:norm-h-flow} to obtain
\begin{multline*}
\pi^{-\frac12}\int_{-\delta_m}^{\delta_m}\int_{\T}\lvert q_{m,n}
\circ\psi_{m,n,t}(\zeta)\rvert \diffs(\zeta)\,\e^{-mt^2}\diff t
\\=\big(1+\Ordo(m^{-\kappa-\frac13})\big)\int_{\calD_{m,n,w_0}} 
\chi_{1} |f_{m,n,w_0}^{\langle \kappa\rangle}(z)
q(z)|\,\e^{-2mR_{\tau,w_0}}\diffA(z)
\le 2m^{-\frac12}\lVert g\rVert_{mQ}.
\end{multline*}
Now if $q(0)=0$, that is, if $g$ vanishes to order $n+1$ or higher
at $w_0$, then 
$\chi_0\kernel_{m,n,w_0}^{\langle \kappa\rangle}$ and $g$ are approximately
orthogonal in $L^2_{mQ}$. 

For further details regarding the above computations, we refer to 
Subsection 4.8 of \cite{HW}. 
\\
\\
\noindent {\sc The $\bar\partial$-correction scheme.}
The approximate normalized reproducing kernels are not globally defined,
and are consequently not elements of our Bergman spaces of entire functions. 
However, by applying the H{\"o}rmander-type $\bar\partial$-estimate of 
Proposition~\ref{prop:horm}, we obtain a solution $u=u_{m,n,w_0}$ to the equation
\[
\bar\partial u=\kernel_{m,n,q_0}^{\langle \kappa\rangle}\bar\partial\chi_{0}
\]
which exponential decay of the norm of $u_{m,n,w_0}$ in $L^2_{mQ}$.
By the proposition, it vanishes to order $n$ at the root point $w_0$, and
has exponentially small norm in $L^2_{mQ}=L^2(\C,\e^{-2mQ})$. 
The function 
\[
\kernel^\star_{m,n,w_0}:=\chi_{0}\kernel_{m,n,w_0}
^{\langle \kappa\rangle}-u_{m,n,w_0}
\]
is also an approximate normalized partial Bergman reproducing kernel 
of the correct accuracy, but this time it at least is an element of the 
right space,
\[
\kernel^\star_{m,n,w_0}\in A^2_{mQ,n,w_0}.
\]
We denote by $\Pop_{m,n+1,w_0}$ the orthogonal projection onto the subspace
$A^2_{mQ,n+1,w_0}$ of functions vanishing to order at least $n+1$, and put
\[
\tilde{\kernel}_{m,n,w_0}=\kernel^\star_{m,n,w_0}
-\Pop_{m,n+1,w_0}\kernel^\star_{m,n,w_0}.
\]
Here, we have the norm estimate
\begin{equation}\label{eq:norm-orth-proj}
\lVert \Pop_{m,n+1,w_0}\kernel^\star_{m,n,w_0}\rVert_{mQ} =\sup_{g\in A^2_{mQ,n+1,w_0}}
\frac{|\langle g,\kernel^\star_{m,n,w_0}\rangle_{mQ}|}{\lVert g\rVert_{mQ}}
=\Ordo(m^{-\kappa-\frac13})
\end{equation}
which shows that the correction is very small.
By construction, $\kernel_{m,n,w_0}^{\langle \kappa\rangle}$ vanishes precisely 
to the order $n$ at the root point $w_0$. Moreover, the small perturbations 
$u_{m,n,w_0}$ and $\Pop_{m,n+1,w_0}\kernel^\star_{m,n,w_0}$ vanish at least 
to order $n$ at $w_0$. 
It follows that $\tilde{\kernel}_{m,n,w_0}$ vanishes precisely to the 
correct order, that is
to say,
\[
\tilde{\kernel}_{m,n,w_0}(z)=C(z-w_0)^n+\Ordo(\lvert z-w_0\rvert^{n+1}),
\]
holds near $w_0$ for some complex constant $C\ne0$. The constant $C$ 
is close to being positive real,
since the $\bar\partial$-correction $u_{m,n,w_0}$ is small. Indeed, we have
$C=(1+\Ordo(\e^{-\alpha_1m}))\,C_1$ where the constant $C_1>0$ may depend 
on all the 
parameters but the parameter $\alpha_1>0$ is a uniform constant. 
Since the function $\tilde{\kernel}_{m,n,w_0}$ is automatically orthogonal 
to $A^2_{mQ,n+1,w_0}$,
it follows that $\tilde{\kernel}_{m,n,w_0}$ equals a scalar multiple of 
the true root function $\kernel_{m,n,w_0}$: 
\[
\tilde{\kernel}_{m,n,w_0}=c\,\kernel_{m,n,w_0},
\]
for some complex constant $c\ne0$. In view of the above, we conclude that 
$c=(1+\Ordo(\e^{-\alpha_1m}))\,c_1$, where $c_1>0$ may depend on all the 
parameters.
As $\tilde{\kernel}_{m,n,w_0}(w_0)$ is approximately real, it follows that
 $c=c'\gamma$, where $c'$ is real and positive, while 
 $\gamma=1+\Ordo(m^{-\kappa-\frac12})$.
It follows from \eqref{eq:norm-orth-proj} that we have the estimate
\[
\big\lVert \tilde{\kernel}_{m,n,w_0}-\chi_{0}\kernel_{m,n,w_0}^{\langle \kappa\rangle}
\big\rVert_{mQ}=
\Ordo(m^{-\kappa-\frac{1}{3}})
\]
and since 
\begin{equation}\label{eq:norm-approx-kern}
\|\chi_0\kernel_{m,n,w_0}^{\langle\kappa\rangle}\|_{mQ}=1+\Ordo(m^{-\kappa-\frac13}) 
\end{equation}
we obtain that positive constant $c_1$ has the asymptotics 
$c_1=1+\Ordo(m^{-\kappa-\frac1{3}})$, which allows to say that 
$\tilde{\kernel}_{m,n,w_0}$ and the true root function $\kernel_{m,n,w_0}$,
which differ by a multiplicative constant, are very close. 
It now follows that
\[
\big\lVert \kernel_{m,n,w_0}-
\chi_{0}\kernel_{m,n,w_0}^{\langle \kappa\rangle}\big\rVert_{mQ}=
\Ordo(m^{-\kappa-\frac{1}{3}}),
\]
so that $\kernel_{m,n,w_0}$ has the desired asymptotic expansion in norm. 
In view of Proposition~\ref{prop:growth}, the pointwise expansion 
is essentially immediate from the $L^2$-estimate, at least in the region 
$\Omega_{\tau,w_0,m}$ where
\[
\mathrm{dist}_{\C}(z,\Omega_{\tau,w_0})\le A\,m^{-\frac12}(\log m)^{\frac12},
\] 
which is where the functions $\breve{Q}_{\tau,w_0}$ and $\hat{Q}_{\tau,w_0}$
are comparable in the sense that
\[
0\le m(\hat{Q}_{\tau,w_0}-\breve{Q}_{\tau,w_0})\le A^2D\log m
\]
for some fixed positive constant $D$ depending only on $Q$.
The only remaining issue is that the error terms are slightly worse 
than claimed. 
However, by replacing $\kappa$ with an integer larger than 
$\kappa+2+A^2D$ and by deriving the expansion 
with the indicated higher accuracy, we conclude that the desired 
error terms may be obtained as well.

Turning to the norm control on the set $\Omega_{\tau,w_0,m}$, we note that 
by elementary Hilbert space methods, we have that
\begin{multline*}
\int_{\Omega_{\tau,w_0,m}}|\kernel_{m,n,w_0}|^2\e^{-2mQ}\diffA
=\int_{\Omega_{\tau,w_0,m}}|\kernel_{m,n,w_0}^{\langle \kappa\rangle}|^2\e^{-2mQ}\diffA
\\+\Ordo\Big(\lVert \kernel_{m,n,w_0}-
\chi_0\kernel_{m,n,w_0}^{\langle \kappa\rangle} \rVert_{2mQ}\Big).
\end{multline*}
We need to calculate the integral on the right-hand side:
\begin{multline*}
\int_{\Omega_{\tau,w_0,m}}|\kernel_{m,n,w_0}^{\langle \kappa\rangle}|^2\e^{-2mQ}\diffA=\\
\int_{\C}|\chi_0\,\kernel_{m,n,w_0}^{\langle \kappa\rangle}|^2\e^{-2mQ}\diffA
- \int_{\mathrm{supp}(\chi_0)\setminus\Omega_{\tau,m,w_0}}
|\chi_0\kernel_{m,n,w_0}^{\langle \kappa\rangle}|^2\e^{-2mQ}\diffA,
\end{multline*}
where we use that $\chi_0=1$ on $\Omega_{m,\tau,w_0}$ provided that 
$m$ is big enough.
In view of \eqref{eq:norm-approx-kern} the first integral on the 
right-hand side equals $1+\Ordo(m^{-\kappa-\frac13})$.
For any $z\in\mathrm{supp}(\chi_0)\setminus\Omega_{\tau,w_0,m}$ we have 
the bound
$2m(Q-\breve{Q}_{\tau,w_0})(z)\ge A^2D\log m$ where $D$ is the positive 
constant encountered previously. 
Consequently, we have the estimate
\begin{multline}
\int_{\mathrm{supp}(\chi_0)\setminus\Omega_m}
|\chi_0\kernel_{m,n,w_0}^{\langle\kappa\rangle}|^2\e^{-2mQ}\diffA
\\
\le C_0m^{\frac12}|\mathrm{supp}(\chi_0)\setminus\Omega_m|_{\mathrm{A}}\e^{-A^2D\log m}
=\Ordo(m^{-A^2D+\frac12}).
\end{multline}
If $A$ is chosen large enough, it follows that
\[
\int_{\Omega_{\tau,w_0,m}}|\kernel_{m,n,w_0}|^2\e^{-2mQ}\diffA=1+\Ordo(m^{-\kappa-\frac13}).
\]
This completes the outline of the proof.
\end{proof} 

\begin{proof}[Proof sketch of Theorem~\ref{thm:main-prel}]
The proof of Theorem~\ref{thm:main-prel} is entirely analogous to the 
above proof of Theorem~\ref{thm:main1},
essentially amounting to putting $\tau=0$ in the latter context. 
In the setting of Theorem~\ref{thm:main-prel}, there exists already a 
forbidden region around the point $w_0$,
and hence permits us to consider $\tau=0$. Indeed, the reason why we 
required that $\tau>0$ in the context of 
Theorem~\ref{thm:main1} was to allow for the instance when the 
off-spectral component $\Omega_{\tau,w_0}$ 
shrinks down to the point $\{w_0\}$ as $\tau\to0$.
\end{proof}

\subsection{Interface asymptotics of the Bergman density}
In this section we show how to obtain the error function transition 
behavior of Bergman densities at interfaces, where the interface may occurs
as a result of a region of negative curvature (understood as where
$\hDelta Q<0$ holds in terms of the potential $Q$) or as a consequence
of dealing with partial Bergman kernels. 
Here, we focus on the the partial Bergman kernel analysis. In fact, we
may think of the first instance of the full Bergman kernel as a special case
and maintain that it is covered by the presented material.

The following Corollary of the main theorem summarizes the asymptotics of 
normalized off-spectral partial Bergman kernels in a suitable form. The domains 
$\Omega_{\tau,w_0,m}$ are as in Theorem~\ref{thm:main1}, for a given positive 
parameter $A$ chosen suitably large.

\begin{cor}\label{cor:main1}
Under the assumptions of Theorem~\ref{thm:main1}, we have the asymptotics
\begin{multline*}
\lvert \kernel_{m,n,w_0}(z)\rvert^2\e^{-2mQ(z)}
\\
=\pi^{-\frac12}m^{\frac12}
\lvert \varphi'_{\tau,w_0}(z)\rvert\,
\e^{-2m(Q-\breve{Q}_{\tau,w_0})(z)}\big\{\e^{2\mathrm{Re}\,\calH_{Q,\tau,w_0}(z)}
+\Ordo(m^{-1})\big\},
\end{multline*}
on the domain $\Omega_{\tau,w_0,m}$, as $n=\tau m\to+\infty$ while 
$\tau\in I_0$,
where $\mathcal{H}_{Q,\tau,w_0}$ is the bounded holomorphic function on 
$\Omega_{\tau,w_0}$ whose real part equals $\frac14\log(2\hDelta Q)$ on the 
boundary, and is real-valued at the root point $w_0$.
\end{cor}

\begin{proof}
In view of the decomposition \eqref{eq:Qbreve101}, this is just the
assertion of Theorem \ref{thm:main1} with accuracy $\kappa=1$.
\end{proof}

We proceed with a sketch of the error function asymptotics at interfaces,
in particular we point out why we may proceed exactly as is done in the 
proof of Theorem~1.4.1 of \cite{HW}.

\begin{proof}[Proof sketch of Corollary~\ref{cor:interface}]
We expand the partial Bergman kernel $K_{m,n,w_0}$ along the diagonal in terms 
of the root functions $\kernel_{m,n',w_0}$, for $n'\ge n$.
We keep $\tau=\frac{n}{m}$ throughout.
In view of Theorem \ref{thm:expansion}, we have
\begin{equation}
\label{eq:sum-density}
K_{m,n,w_0}(z_m(\xi),z_m(\xi))\,\e^{-2mQ(z_m(\xi))}=\sum_{n'=n}^{+\infty}
\lvert \kernel_{m,n',w_0}(z_m(\xi))\rvert^2\e^{-2mQ(z_m(\xi))},
\end{equation}
where $z_0\in\partial\Omega_{\tau,w_0}$ and where $z_m(\xi)$
gives the rescaled coordinate implicitly by
\[
z_m(\xi)=z_0+\nu\frac{\xi}{\sqrt{2m\hDelta Q(z_0)}}.
\]
The rescaled Bergman density is then obtained by
\[
\varrho_m(\xi)=\frac{1}{2m\hDelta Q(z_0)}
\sum_{n\ge \epsilon m}\lvert\kernel_{m,n,w_0}(z_m(\xi))\rvert^2\e^{-2mQ(z_m(\xi))}.
\]
In view of the assumed $(I_0,w_0)$-admissibility, we may apply the 
asymptotic expansion in the main result, specifically in the form of 
Corollary~\ref{cor:main1}.
Since Proposition~\ref{prop:free-bdry} tells us how the smooth Jordan curves 
$\partial\Omega_{\tau,w_0}$ propagate, a Taylor series expansion
of the function $Q-\breve{Q}_{\tau,w_0}$ allows us to write
the partial Bergman density approximately as a sum of translated Gaussians
\begin{equation}
\label{eq:sum-density1}
\varrho_m(\xi)=\frac{1}{\sqrt{2\pi}}
\sum_{j\ge 0}\frac{\gamma_0}{\sqrt{m}}
\e^{-\frac12(2\Re \xi + j\frac{\gamma_0}{\sqrt{m}})^2}
+\Ordo\big(m^{-\frac12}(\log m)^3\big),
\end{equation}
where $\gamma_0=\gamma_{z_0,w_0,Q}$ is a positive constant. 
As in the proof of Theorem~1.4.1 of \cite{HW}, we proceed to 
interpret the above sum \eqref{eq:sum-density1} as a Riemann sum for the 
integral formula for the error function:
\[
\mathrm{erf}(2\mathrm{Re}\, \xi)=\frac{1}{\sqrt{2\pi}}\int_{0}^{\infty}
\e^{-\frac12(2\mathrm{Re}\,(\xi)+t)^2}\diff t.
\]
This proof is complete.
\end{proof}

\section{The foliation flow for more general area forms}
\label{s:off-spectral-general}
\subsection{More general area forms}
It will be desirable to obtain some flexibility on the part 
of the weight $\e^{-2mQ}$
in the expansion of Theorem~\ref{thm:main1}. In particular, in the following 
subsections we will discuss various situations in which one needs 
asymptotics for root functions and orthogonal polynomials with respect to
measures
\[
\e^{-2mQ}V\,\diffA,
\]
where $V$ is a positive $C^2$-smooth function which is real-analytic 
in a neighborhood of the fixed smooth spectral interface
of interest, which meet the polynomial growth bound
\begin{equation}
C_1(1+|z|^2)^{-N} \le V(z)\le C_2(1+|z|^2)^{N},
\label{eq:polgrowthV}
\end{equation}
for some positive constants $C_1$ and $C_2$ and some fixed integer $N<+\infty$.
We also require that for some positive constant $C_3$, it holds that
\begin{equation}\label{eq:subharm-V}
\hDelta \log V(z)\le C_3\hDelta Q(z),\qquad z\in\C.
\end{equation}
In particular,
this covers working with the spherical area measure 
\[
\diffA_{\mathbb{S}}(z):=(1+|z|^2)^{-2}\diffA(z)
\] 
in place of planar area measure simply by considering $V(z)=(1+|z|^2)^{-2}$. 
Working with the spherical area measure has the 
advantage of invariance with respect to rotations and inversion.
For a more general conformal factor $V$, we factor 
$V\diffA=V_{\mathbb{S}}\diffA_{\mathbb{S}}$,
where $V_{\mathbb{S}}(z)=(1+|z|^2)^2V(z)$, and see that our weighted measure is
\[
\e^{-2mQ}V_{\mathbb{S}}\,\diffA_{\mathbb{S}},
\]
which has a more invariant appearance. If we write $\nu(z)=z^{-1}$, 
the spaces of polynomials of degree at most $n$ with respect to the $L^2$-space
with measure $\e^{-2mQ}V_{\mathbb{S}}\,\diffA_{\mathbb{S}}$ becomes isometrically 
isomorphic to the $L^2$-space of rational functions on the sphere $\mathbb{S}$
with a pole of order at most $n$ at the origin, with respect to the
$L^2$-space with measure $\e^{-2mQ\circ\nu}V_{\mathbb{S}}\circ\nu\,
\diffA_{\mathbb{S}}$. This provides an extension of the scale of root functions
to zeros of negative order (i.e. poles), and the apparent similarities between
orthogonal polynomials and root functions may be viewed in this light.  
This analogy goes even deeper than that. Assuming that $0$ is an off-spectral
point for the weighted $L^2$-space with measure 
$\e^{-2mQ\circ\nu}V_{\mathbb{S}}\circ\nu\,\diffA_{\mathbb{S}}$, we may multiply
by a suitable power of the conformal mapping from the off-spectral region 
to the unit disk $\D$, which preserves the origin, to obtain a space of 
functions holomorphic in a neighborhood of the off-spectral region. 
H\"ormander-type estimates for the $\bar\partial$-equation then permit us
to correct the functions so that they are entire, with small cost in norm.  

We note that more general area forms appear naturally from working with 
perturbations of the 
potential $Q$. Indeed, if we consider $\tilde Q=Q-m^{-1}h$ for
some smooth function $h$ of modest growth, we have that 
\[
\e^{-2m\tilde Q}\diffA=\e^{-2m Q}\e^{2h}\diffA,
\] 
which corresponds precisely to the conformal factor $V=\e^{2h}$.

\subsection{The asymptotics of root functions and 
orthogonal polynomials for more general area forms}
Our analysis will show that the root function asymptotics of 
Theorem~\ref{thm:main1} holds also in the context of a general area form, 
with only a slight change in the structure of the coefficients 
$\calB_{j,\tau,w_0}$.
Let $A^2_{mQ,V}$ denote the weighted Bergman space of entire functions with
respect to the Hilbert space norm
\[
\|f\|_{mQ,V}^2:=\int_\C|f|^2\e^{-2mQ}V\diffA<+\infty.
\]
The corresponding Bergman kernel is denoted by $K_{m,V}$. We also need the 
partial Bergman spaces $A^2_{mQ,V,n,w_0}$, consisting of the functions in
$A^2_{mQ,V}$ that vanish at $w_0$ to order $n$ or higher. These are closed 
subspaces of $A^2_{mQ,V}$ which get smaller as $n$ increases: 
$A^2_{mQ,V,n+1,w_0}\subset A^2_{mQ,V,n,w_0}$. The successive difference spaces
$A^2_{mQ,V,n,w_0}\ominus A^2_{mQ,V,n+1,w_0}$ have dimension at most $1$. If the 
dimension equals $1$, we single out an element $\kernel_{m,n,w_0,V}\in 
A^2_{mQ,V,n,w_0}\ominus A^2_{mQ,V,n+1,w_0}$ of norm $1$, which has positive derivative
of order $n$ at $w_0$. In the remaining case when the dimension equals $0$ we
put  $\kernel_{m,n,w_0,V}=0$. As before, we call $\kernel_{m,n,w_0,V}$ 
\emph{root functions}, and observe that these are the same objects we 
defined earlier for $V=1$ in terms of an extremal problem.

\begin{thm}\label{thm:twist-root}
Under the assumptions of Theorem \ref{thm:main1} and the above-mentioned
assumptions on $V$, with respect to the interface $\partial\Omega_{\tau,w_0}$,
we have, using the notation of the same theorem, for fixed accuracy and 
a given positive real $A$, the asymptotic expansion of the root function  
\begin{multline*}
\kernel_{m,n,w_0,V}(z)\\
=m^{\frac14}(\varphi_{\tau,w_0}'(z))^{\frac12}(\varphi_{\tau,w_0}(z))^n
\e^{m\calQ_{\tau,w_0}}
\Big\{\sum_{j=0}^{\kappa}m^{-j}\calB_{j,\tau,w_0,V}(z)
+\Ordo\big(m^{-\kappa-1}\big)\Big\},
\end{multline*}
on the domain $\Omega_{\tau,w_0,m}$ which depends on $A$, 
where $\tau=\frac{n}{m}$, and the implied constant is uniform.
Here, the main term $\calB_{0,\tau,w_0}$ is zero-free and smooth up to the 
boundary on $\Omega_{\tau,w_0}$, positive at $w_0$, with prescribed modulus
\[
[\calB_{0,\tau,w_0,V}(z)|=\pi^{-\frac14}(\hDelta Q(z))^{\frac14} V(z)^{-\frac12},\qquad
z\in\partial\Omega_{\tau,w_0}.
\]
\end{thm}

The proof of this theorem is analogous to that of Theorem \ref{thm:main1},
given that we have explained how to modify the orthogonal foliation flow with
respect to the general area form in Lemma \ref{lem:main-flow-general}. 
The lemma is applied with $s=m^{-1}$. We omit the necessary details.

We turn next to the computation of the coefficients $\calB_{j,\tau,w_0,W}$ 
in the above
expansion. 
We recall that $R_{\tau,w_0}$ is the potential 
induced by $Q$ in the canonical positioning procedure, and we put analogously 
\[
W_{\tau,w_0}(z)=V\circ\varphi_{\tau,w_0}^{-1}(z),\qquad z\in\D(0,\hdelta).
\]
For the formulation, we need the orthogonal projection $\Pop_{H^2_{0}}$ of 
$L^2(\T)$ onto the Hardy space $H^2_0$ of functions $f$ in the Hardy space 
$H^2$ that vanish at the origin.

\begin{thm}\label{thm:comput-coeff}
In the asymptotic expansion of root functions in 
Theorem~\ref{thm:twist-root}, the coefficient functions $\calB_{j,\tau,w_0,V}$ 
are obtained by 
\[
\calB_{j,\tau,w_0,V}=(\varphi_{\tau,w_0}')^{\frac12}B_{j,\tau,w_0,V}\circ
\varphi_{\tau,w_0},\qquad j=1,2,3,\ldots.
\]
If $H_{\tau,w_0,V}$ denotes the unique bounded holomorphic function on $\D$, 
whose real part meets
\[
\Re H_{\tau,w_0,V}=\frac{1}{4}\log(4\hDelta R_{\tau,w_0}) + 
\frac12\log(W_{\tau,w_0}), \qquad \text{on } \T,
\]
with $\Im H_{\tau,w_0,V}(0)=0$, the functions $B_{j,\tau,w_0,V}$ 
may be obtained algorithmically as
$$
B_{j,\tau,w_0,V}=c_{j} \e^{H_{\tau,w_0,V}} - \e^{H_{\tau,w_0,V}}
\Pop_{H^2_{0}}\big[\e^{\bar{H}_{\tau,w_0,V}}F_{j}\big]
$$
for some real constants $c_{j}=c_{j,\tau,w_0,V}$ and real-analytically smooth 
functions $F_j=F_{j,\tau,w_0,V}$ on the unit circle $\T$. Here, both the 
constants $c_j$ and the functions $F_j$ may be computed iteratively in 
terms of $B_{0,\tau,w_0,V},\ldots B_{j-1,\tau,w_0,V}$.
\end{thm}

One may further derive concrete expressions for the constants $c_j$
and the real-analytic functions $F_j$ in the above result, 
in terms of the rather complicated explicit differential operators
$\Lop_k$ and $\Mop_k$ as defined in equation (1.3.4) and Lemma 3.2.1 in 
\cite{HW}. We should mention that the definition of the operator $\Mop_k$ 
contains a parameter $l$, which is allowed to assume only non-negative 
values. However, the same definition works also for $l< 0$, which is 
necessary for the present application. In terms of the operators
$\Mop_k$ and $\Lop_k$, we have
\[
F_j(\theta)=\sum_{k=1}^{j}\Mop_{k}[B_{j-k,\tau,w_0,V} W_{\tau,w_0}]
\]
and 
\[
c_j=-\frac{(4\pi)^{-\frac14}}{2}\!\!\!\sum_{(i,k,l)\in\indset_{j}} 
\int_{\T}\frac{\Lop_{k}
[r B_{i,\tau,w_0,V}(r\e^{\imag\theta}) \bar{B}_{l,\tau,w_0,V}(r\e^{\imag\theta}) 
W_{\tau,w_0}(r\e^{\imag\theta})]}{(4\hDelta R_{\tau,w_0}
(r\e^{\imag\theta}))^{\frac12}}
\big\vert_{r=1}\diffs(\e^{\imag\theta}).
\]
Here, the index set $\indset_j$ is defined as 
\[
\indset_j=\{(i,k,l)\in\mathbb{Z}_{+,0}^3:i,l<j,\,\,\,i+k+l=j\},
\] 
where we use the notation $\mathbb{Z}_{+,0}:=\{0,1,2,\ldots\}$.
This theorem is obtained in the same fashion as Theorem~1.3.7 in \cite{HW}
in the context of orthogonal polynomials, and we do not write down
a proof here. 

\subsection{The flow modified by a conformal factor}

We proceed first to modify the book-keeping slightly by formulating an analogue
of Definition~\ref{def:class-W}, which applies to weights after canonical 
positioning.

\begin{defn}\label{def:class-W-mod}
Let $\hdelta$ and $\sigma$ be given positive numbers, with $\hdelta>1$.
A pair $(R,W)$ of non-negative $C^2$-smooth weights defined on $\D(0,\hdelta)$ 
is said to belong to the class 
$\mathfrak{W}_\circledast(\hdelta,\sigma)$ if $R\in \mathfrak{W}(\hdelta,\sigma)$
and if the weight $W$ meets the following conditions:
\begin{enumerate}[(i)]
\item $W$ is real-analytic and zero-free in the neighborhood 
$\mathbb{A}(\hdelta^{-1},\hdelta)$ of the unit 
circle $\T$, 
\item The polarization $W(z,{w})$ of $W$ extends to a bounded
holomorphic function of $(z,\bar w)$ on the $2\sigma$-fattened 
diagonal annulus $\hat{\mathbb{A}}(\sigma, \hdelta)$, which is also 
bounded away from $0$.
\end{enumerate}
A collection $S$ of pairs $(R,W)$ is said to be a uniform family
in $\mathfrak{W}_\circledast(\hdelta,\sigma)$ if the weights $R$ with 
$(R,W)\in S$ are confined to a uniform family in 
$\mathfrak{W}(\hdelta,\sigma)$, while $W(z,{w})$
is uniformly bounded and bounded away from $0$ in 
$\hat{\mathbb{A}}(\hdelta,\sigma)$.
\end{defn}

Fix a pair $(\hdelta_0,\sigma_0)$. Just as before, we let $(\hdelta_1,\sigma_1)$
denote a possibly more restrictive pair of positive reals with $\hdelta_1>1$
such that the relevant polarizations are hermitian-holomorphic and uniformly
bounded on $\hat{\mathbb{A}}(\hdelta_1,\sigma_1)$.
In connection with this definition, we recall the bound
\[
\hdelta_1\le\sqrt{1+\sigma_1^2}+\sigma_1,
\] 
which guarantees that that if $f(z,{w})$ is holomorphic in 
$(z,\bar w)$ on the set $\hat{\mathbb{A}}(\sigma_1, \hdelta_1)$, then the function 
$f_{\T}(z)=f(z,\bar{z}^{-1})$ may be continued holomorphically to the annulus
$\mathbb{A}(\hdelta_1^{-1},\hdelta_1)$. 

We proceed with the main result of this section.

\begin{lem}\label{lem:main-flow-general}
Fix an accuracy parameter $\kappa$ and let 
$(R,W)\in \mathfrak{W}_\circledast(\hdelta_0,\sigma_0)$.
Then there exist a radius $\hdelta_2$ with $\hdelta_1>\hdelta_2>1$, 
bounded holomorphic functions $f_{s}$ on $\D(0,\hdelta_1)$ of the form
\[
f_{s}=\sum_{j=0}^{\kappa}s^{j}
B_{j}+\Ordo(s^{\kappa+1}),\qquad z\in\D(0,\hdelta'),
\]
and normalized conformal mappings $\psi_{s,t}$ on $\D(0,\hdelta_2)$ given by 
\[
\psi_{s,t}=\psi_{0,t}+\sum_{\substack{(j,l)\in\indsett_{2\kappa+1}\\ j\ge 1}}s^j t^l 
\hat{\psi}_{j,l}
\] 
such that for $s,t$ small enough it holds that the domains 
$\psi_{s,t}\big(\D\big)$ increase with
$t$, while they remain contained in $\D(0,\hdelta_1)$. 
Moreover, for $\zeta\in\T$, we have
\begin{multline}\label{eq:flow-eq-W}
\lvert f_{s}\circ \psi_{s,t}(\zeta)\rvert^2\,\e^{-2s^{-1}R\circ\psi_{s,t}}
\Re\big(\bar{\zeta}\partial_t\psi_{s,t}(\zeta)
\overline{\psi_{s,t}'(\zeta)}\big)
W\circ\psi_{s,t}(\zeta)
\\=\e^{-s^{-1}t^2}\Big\{(4\pi)^{-\frac12}
+\Ordo\big(\lvert s\rvert^{\kappa+\frac12}
+\lvert t\rvert^{2\kappa+1}\big)\Big\}.
\end{multline}
For small positive $s$, when $t$ varies in the interval $[-\beta_s,\beta_s]$ 
with $\beta_s:=s^{1/2}\log \frac{1}{s}$, the flow of loops 
$\{\psi_{s,t}(\T)\}_t$ cover a neighborhood of the circle $\T$ of width 
proportional to $\beta_s$ smoothly.
In addition, the main term $B_{0}$ is zero-free, positive at the origin, 
and has modulus $|B_0|=\pi^{-\frac14}(\hDelta R)^{\frac14}W^{-\frac12}$ on $\T$, 
and the other terms $B_{j}$ are all real-valued at the origin. 
The implied constant in \eqref{eq:flow-eq-W} is uniformly bounded, provided 
that 
$(R,W)$ is confined to a uniform family of 
$\mathfrak{W}_\circledast(\hdelta_0,\sigma_0)$.
\end{lem}

In order to obtain this lemma, we need to modify the algorithm which gives 
the original result. We proceed to sketch an outline of this modification. 
The omitted details are available in \cite{HW}, and we try to guide the reader 
for easy reading.

We recall the following index sets from \cite{HW}.
For an integer $n$, we introduce
\begin{equation}
\label{eq:index-set}
\indsett_n=\big\{(j,l)\in\mathbb{Z}_{+,0}^2\,:\;2j+l\le n\big\}.
\end{equation}
We observe that if $(j,l)\in\indsett_n$, then $2j\le n$, and that we have
the equivalence 
\begin{equation}
\label{eq:index-set2.0}
(j,l)\in\indsett_{n+1}\,\,\,\,\text{and}\,\,\,\, j\ge1
\,\,\,\Longleftrightarrow\,\,\,
(j-1,l+1)\in\indsett_{n}\,\,\,\,\text{and}\,\,\,l\ge0.
\end{equation} 
We endow the set $\indsett_n$ with the ordering $\prec$ induced by 
the lexicographic ordering, so we agree that $(j,l)\prec(a,b)$ if $j<a$ 
or if $j=a$ and $l<b$.

\begin{proof}[Proof of Lemma~\ref{lem:main-flow-general}]
The conformal mappings $\psi_{s,t}$ are assumed to have the form
\[
\psi_{s,t}=\psi_{0,t}+
\sum_{\substack{(j,l)\in\indsett_{2\kappa+1}\\ j\ge 1}}
s^{j}t^l\hat{\psi}_{j,l}
\]
for some bounded holomorphic coefficients $\hat{\psi}_{j,l}$ and a conformal
mapping
\[
\psi_{0,t}=\sum_{l=0}^{+\infty}t^l \hat{\psi}_{0,l}.
\] 
We make the following initial observation.
In the limit case $s=0$, the flow equation \eqref{eq:flow-eq-W} 
of the lemma forces $\psi_{0,t}$ to be a 
mapping from $\D$ onto the interior of suitably chosen level curves of 
$R$, and from this we may obtain the coefficients $\hat{\psi}_{0,l}$. 
Indeed, if we take logarithms of both sides of the equation
and multiply by $s$ we obtain
\begin{equation}
\label{eq:logflow}
s\log \lvert f_s\circ\psi_{s,t}\rvert^2 -2R\circ\psi_{s,t} 
+ s\log (1-t)\log J_\Psi 
+ s\log  W\circ\psi_{s,t}=-t^2+\Ordo(s),
\end{equation}
where the $J_\Psi$ denotes the associated Jacobian
\[
J_{\Psi_s}((1+t)\zeta):=\Re\big(\bar{\zeta}\partial_t\psi_{s,t}
\overline{\psi_{s,t}'(\zeta)}\big),\qquad \zeta\in\T,
\]
for $t$ in the interval $[-\beta_s,\beta_s]$, so that $J_{\Psi_s}$ gets defined
on the annulus $\mathbb{A}(1-\beta_s,1+\beta_s)$, provided that the coefficient 
functions which define $\psi_{s,t}$ can be found.  
Assuming some reasonable stability with respect to the variable $s$ as 
$s\to0^+$ in \eqref{eq:logflow}, we obtain in the limit that
\begin{equation}
\label{eq:logflow2}
2R\circ\psi_{0,t}(\zeta)=t^2,\qquad \zeta\in\T.
\end{equation}
In particular, the loop $\psi_{0,t}(\T)$ is a part of the level set 
where $R=\frac{1}{2}t^2$. This level set consists of two disjoint simple 
closed curves, one on either side of $\T$, at least for small enough $t$
and locally near $\T$. 
For $t>0$, we choose the curve outside the unit circle, while for $t<0$
we choose the other one. We normalize the mapping $\psi_{0,t}$ so that it
preserves the origin and has positive derivative there.
In this fashion, the coefficients $\hat{\psi}_{0,l}$ get determined uniquely 
by the level set condition. 
We note that the smoothness of the level curves was worked out in some 
detail in Proposition~4.2.5 in \cite{HW}.
Moreover, the coefficient functions $\hat{\psi}_{0,l}$ are given in terms 
of Herglotz integrals as in Proposition~4.6.1 of \cite{HW}, with the 
obvious modifications.

Our next task is to obtain iteratively the coefficients 
$B_{j}$ for $j=0,1,2,\ldots$ and the higher order corrections to the
conformal mapping, given in terms of the coefficients 
$\hat{\psi}_{j,l}$ for $j=1,\ldots,\kappa$. 
It turns out to be advantegeous to work instead with $h_s=\log f_s$ as 
the basic object of study, and make the ansatz
\begin{equation}\label{eq:def-h}
h_s(z)=\sum_{j=0}^{\kappa} s^j b_j(z),\qquad z\in\D(0,\hdelta_1),
\end{equation}
where $b_j$ are bounded holomorphic functions on $\D(0,\hdelta_1)$
with $\Im b_j(0)=0$. It follows that the coefficient function $B_j$ may
be expressed as a multivariate polynomial in the coefficients 
$b_0,\ldots,b_j$, for each $j=0,\ldots,\kappa$.  
The coefficient functions $b_j$ are obtained by differentiating the 
logarithm of the flow equation \eqref{eq:flow-eq-W}. 
To set things up correctly, we write
\[
\omega_{s,t}(\zeta)=\lvert f_s\circ\psi_{s,t}\rvert^2 
\e^{-2s^{-1}(R\circ\psi_{s,t}-\frac{1}{2}t^2)}
(W\circ\psi_{s,t})\Re\big(\bar{\zeta}\partial_t\psi_{s,t}
\overline{\psi_{s,t}'}
\big)
\]
for $\zeta\in\T$, where we recall that $R\circ\psi_{0,t}=\frac{1}{2}t^2$,
and put
\begin{multline}\label{eq:def-varpi}
\varpi_{s,t}=\log\omega_{s,t}
=2\Re h_s\circ\psi_{s,t}-
\frac{2}{s}\big(R\circ\psi_{s,t}-\tfrac12 t^2\big)
\\
+\log\Re(\bar\zeta\partial_t\psi_{s,t}\overline{\psi_{s,t}'})+
\log W\circ\psi_{s,t}.
\end{multline}
We need to show that 
\begin{equation}\label{eq:flow-omega}
\varpi_{s,t}(\zeta)=
-\frac12\log(4\pi)+\Ordo(\lvert s\rvert^{\kappa+\frac12}
+\lvert t\rvert^{2\kappa+1}),
\qquad \zeta\in\T
\end{equation}
when $s\to0$ while $|t|\le s^{\frac12}\log \frac{1}{s}$.
Since $\varpi_{s,t}$ should be smooth in the parameters $s$ and $t$, 
by the multivariate Taylor formula, this is equivalent to having the system of 
equations
\begin{multline}
\label{eq:main-flow}
\hat{\varpi}_{j,l}(\zeta)
=\frac{\partial_{s}^j\partial_t^l\varpi_{s,t}(\zeta)}{j!l!}
\Big\vert_{s=t=0}
\\
=
\begin{cases}
-\frac12\log(4\pi)&\text{ for }\zeta\in\T \text{ and } (j,l)=(0,0),\\
0& \text{ for }\zeta\in\T \text{ and } 
(j,l)\in\indsett_{2\kappa}\setminus(0,0),
\end{cases}
\end{multline}
fulfilled. 
If, in addition, we can show that the functions $b_j$ and $\hat{\psi}_{j,l}$
remain holomorphic and uniformly bounded in the appropriate domains 
provided that $(R,W)$ remains confined to a uniform family of 
$\mathfrak{W}_\circledast(\hdelta_0,\sigma_0)$, the result follows.

\medskip

\noindent{\sc How to solve for the unknown coefficient functions.} 
We proceed to solve the system \eqref{eq:main-flow}. The approach is to 
first express the Taylor coefficients $\hat{\varpi}_{j,l}$ in terms of our 
unknown coefficient functions $b_0,\ldots,b_\kappa$ and $\hat{\psi}_{j,l}$ for 
$(j,l)\in\indsett_{2\kappa+1}$, and then, as a second step, to insert the 
system of equations \eqref{eq:main-flow}. 
We determine the unknowns by an iterative procedure. At a given step in
the iteration, some of the coefficient functions will be already found.
We split the equation \eqref{eq:main-flow} into a term containing 
precicely one unknown coefficient function and a second term which contains 
only already determined coefficient functions. Here, we refrain from giving 
a complete account, which is available in \cite{HW} modulo minor modifications
to fit the present setup. 
The higher order Taylor coefficients of the function $\varpi_{s,t}$ with 
respect of $s$ and $t$ are given as follows. When $j\ge 0$ and $l\ge 1$, we 
have
\begin{equation}
\label{eq:hatphi-det}
0=\hat{\varpi}_{j,l}(\zeta)
=-2(4\hDelta R(\zeta))^{\frac12}\,\Re(\bar\zeta\hat{\psi}_{j+1,l-1}(\zeta))
+\frakF_{j,l,W}(\zeta),
\qquad \zeta\in\T,
\end{equation}
while for $j\ge 1$ and $l=0$, we have
\begin{equation}\label{eq:Bj-det}
0=\hat{\varpi}_{j,0}(\zeta)
=2\Re b_j(\zeta)
+\frakF_{j,0,W}(\zeta),\qquad \zeta\in\T.
\end{equation}
Here, we inserted the equation \eqref{eq:main-flow}
for added convenience. The expressions $\frakF_{j,l,W}$ are real-valued 
real-analytic functions which are uniformly bounded while $(R,W)$ 
remains in a uniform family 
in $\mathfrak{W}_\circledast(\hdelta_0,\sigma_0)$, and may be explicitly 
written down using the multivariate Fa{\`a} di Bruno's formula. 
The crucial point for us is the dependence structure of the functions 
$\frakF_{j,l,W}$, which remains the same as in the algorithm for the 
orthogonal polynomials: 

\medskip

\noindent {\em {\rm($\frakF$-i)} 
For $l\ge 1$, the function $\frakF_{j,l,W}$ is an expression 
in terms of the functions $b_0,\ldots,b_j$ as well $\hat{\psi}_{p,q}$ for 
indices $(p,q)\in\indsett_{2\kappa+1}$ with $(p,q)\prec (j+1,l-1)$, and also 
involves $R$ and $W$,}

\medskip

\noindent whereas 

\medskip

\noindent {\em {\rm($\frakF$-ii)} For $l=0$, the function $\frakF_{j,0,W}$ 
is an expression in terms of $b_0,\ldots,b_{j-1}$ and 
$\hat{\psi}_{p,q}$ for indices $(p,q)\in\indsett_{2\kappa+1}$ with 
$(p,q)\prec(j+1,0)$, and also involves $R$ and $W$.}

\medskip

The dependence is basically multivariate polynomial dependence, with the
need to allow also for partial derivatives of some of the given expressions.
This aspect is described in great detail in terms of the polynomial 
complexity classes introduced in Section~4.8 of \cite{HW}.
That ($\frakF$-i)-($\frakF$-ii) hold follows by noticing that Propositions 
4.2.5, 4.4.1 and 4.6.1 in \cite{HW} remain basically unchanged, 
while Propositions~4.7.1, 4.9.1, 4.10.1 and 4.12.1 in \cite{HW} require
 the obvious modifications related to our replacing the weight 
$\e^{-2mR}$ on the exterior disk $\D_{\e}(0,\rho)$ by a weight $\e^{-2mR} W$
on the disk $\D(0,\rho)$.
In particular, it is important that the real-analytic function $W$ is strictly 
positive in a fixed neighborhood of the unit circle $\T$, 
so that $\log W$ is a real-analytic function
in the same region.
A natural approach to the computations is to write 
\[
\varpi^{\mathrm{I}}_{s,t}=2\Re h_s\circ\psi_{s,t}
-\frac{2}{s}\big(R\circ\psi_{s,t}-\tfrac12 t^2\big)
+\log\big(\bar\zeta\partial_t\psi_{s,t}\overline{\psi_{s,t}'(\zeta)}\big)
\]
which is essentially the expression which gets expanded in \cite{HW},
and introduce the modification $\varpi_{s,t}^{\mathrm{II}}=\log W\circ\psi_{s,t}$, 
since then 
$\varpi_{s,t}=\varpi_{s,t}^{\mathrm{I}}+\varpi_{s,t}^{\mathrm{II}}$.
If we introduce the notation
\[
\hat{\varpi}_{j,l}^{\mathrm{I}}=
\frac{\partial_{s}^j\partial_t^l\hat{\varpi}_{j,l}^{\mathrm{I}}}{j!l!}
\Big\vert_{s=t=0},
\qquad \hat{\varpi}_{j,l}^{\mathrm{II}}
=\frac{\partial_{s}^j\partial_t^l\hat{\varpi}_{j,l}^{\mathrm{II}}}{j!l!}
\Big\vert_{s=t=0},
\]
it follows immediately that $\hat{\varpi}_{j,l}^{\mathrm{I}}$ are as in 
Proposition~4.12.1 of \cite{HW}, while the latter may be computed explicitly
using the multivariate Fa{\`a} di Bruno formula \cite{bruno}. The leading 
behavior of $\hat{\varpi}_{j,l}$ comes from the contribution of 
$\hat{\varpi}_{j,l}^{\mathrm{I}}$, which is as in \cite{HW},
while it is easily verified that $\hat{\varpi}_{j,l}^{\mathrm{II}}$ for $(j,l)\in
\indsett_{2\kappa}$ is an expression
in terms of the coefficient functions $\hat{\psi}_{p,q}$ with $p\le j$ and 
$q\le l$ and of some partial derivatives of $\log W$. But then in particular
$(p,q)\in\indsett_{2\kappa}$ with $(p,q)\prec(j,l+1)$.
It is now immediate that the remainders $\frakF_{j,l,W}$ have the indicated 
properties.

\medskip

\noindent{\sc Sketch of the algorithm.}
We sketch the solution algorithm below, and indicate where 
the necessary modifications to the corresponding steps in \cite{HW} 
are required.
We need the  notation $\Gamma_t$ for the loop with  
\[
R(z)=\tfrac{1}{2}t^2,\qquad z\in\Gamma_t.
\]
The loop is a perturbation of the unit circle $\T$ for $t$ close to $0$. 
However, for $t\ne0$ there are two such loops, one which is inside the 
circle $\T$, and one which is outside. For $t>0$ we choose the outside loop, 
whereas for $t<0$ we instead choose the inside loop. This way, the 
domain enclosed by $\Gamma_t$ grows with $t$.  

\medskip

\noindent{\sc Step 1.} Take as above $\psi_{0,t}$ to be the conformal 
mapping $\psi_{0,t}:\D\to D_{t}$, where $\psi_{0,t}(0)=0$ and $\psi_{0,t}'(0)>0$,
and $D_{t}$ denotes the bounded domain enclosed by the curve $\Gamma_t$.
It follows from the smoothness of the flow of the level curves $\Gamma_t$ 
(for details see Proposition~4.2.5 in \cite{HW}) that we have an expansion 
\[
\psi_{0,t}=\sum_{l=0}^{+\infty}t^l\hat{\psi}_{0,l},
\]
which determines the coefficient functions $\hat{\psi}_{0,l}$ 
for all $l=0,1,2,\ldots$ (see Proposition~4.6.1 \cite{HW}).
In particular, we have 
\begin{equation}\label{eq:def-hat-psi-01}
\Re(\bar\zeta \hat{\psi}_{0,1}(\zeta))=(4\hDelta R(\zeta))^{-\frac12},\qquad
\zeta\in\T.
\end{equation}

\medskip

\noindent{\sc Step 2.} The equation \eqref{eq:main-flow} 
for $(j,l)=(0,0)$ together with \eqref{eq:def-hat-psi-01} gives that
\[
2\Re b_0(\zeta)-\frac12\log(4\hDelta R(\zeta))\,
+\log W(\zeta)=-\frac12\log (4\pi),\qquad \zeta\in\T,
\]
which in its turn determines $b_{0}$. Indeed, the only holomorphic 
function which is real-valued at the origin and meets this equation is 
\begin{equation}
\label{eq:b0formula}
b_{0}=
-\frac14\log(4\pi)+\frac14\Hop_{\D}\big[\log(4\hDelta R)-2\log W\big]
\quad\text{on }\,\,\T,
\end{equation}
where the Herglotz operator $\Hop_{\D}$ is given by
\[
\Hop_\D [f](z):=\int_\T \frac{1+\bar w z}{1-\bar w z}\,f(w)\,\diffs(w),\qquad
z\in\D.
\]
The definition of the Herglotz operator is extended to the boundary $\T$ via 
nontangential boundary values, and, whenever possible, across the circle $\T$
by analytic continuation. 
As a consequence of the assumptions on the potential $Q$ and the conformal 
factor $V$ as well as the regularity of the conformal mapping 
$\varphi_{\tau,w_0}$, the function $W$ is real-analytic and positive 
in a neighborhood of $\T$, and moreover the pair $(R,W)$
meets the regularity requirements of Definition \ref{def:class-W-mod}.
This means that essentially we are in the same setting as explained in 
\cite{HW}. 
For instance, we may conclude that the function $b_{0}$ given by 
\eqref{eq:b0formula} extends as a bounded holomorphic function on a disk
$\D(0,\hdelta_1)$ with radius $\hdelta_1>1$.

\medskip

\noindent We proceed to Step 3 with $j_0=1$. We note that 
$(j,l)\in\indsett_{2\kappa+1}$ entails that $0\le j\le\kappa$.

\medskip

\noindent{\sc Step 3.} 
To begin with, we have an integer $1\le j_0\le\kappa$ 
for which we have already successfully determined the coefficient functions 
$b_j$ for $0\le j\le j_0-1$ as well as $\hat{\psi}_{j,l}$ for all 
$(j,l)\in\indsett_{2\kappa+1}$ with $(j,l)\prec(j_0,0)$. Note that is known
to be so for $j_0=1$, by Steps 1 and 2. 
In this step, we intend to determine all the coefficient functions 
$\hat{\psi}_{j,l}$ with $(j,l)\in\indsett_{2\kappa+1}$ and 
$(j,l)\prec(j_0+1,0)$, and keep track of the equations 
\eqref{eq:main-flow} with $(j,l)\in\indsett_{2\kappa}$ that get solved
along the way. The induction hypothesis includes the assumption that 
the system of equations \eqref{eq:main-flow} holds for all 
$(j,l)\in\indsett_{2\kappa}$ with $l\ge1$ and $(j,l)\prec(j_0-1,1)$ (this
is vacuous for $j_0=1$), as well as for $(j,l)$ with $0<j<j_0$ and $l=0$.
 
We need only find $\hat{\psi}_{j,l}$ for $(j,l)\in\indsett_{2\kappa+1}$ 
with $j=j_0$. We do this by induction in the parameter $l$, 
starting with $l=0$.
Assume for the moment that it has been carried out for all $l=0,\ldots,l_0-1$.
The coefficient $\hat{\psi}_{j_0,l_0}$ which we are looking for appears as 
the leading term in the equation \eqref{eq:hatphi-det} corresponding to 
$(j,l)=(j_0-1,l_0+1)$, which reads
\[
-2(4\hDelta R)^{\frac12}\Re(\bar{\zeta}\hat{\psi}_{j_0,l_0})+\frakF_{j_0-1,l_0+1}=0
\quad\text{on }\,\T,
\]
where $\frakF_{j_0-1,l_0+1}$ is an expression in the already determined
coefficient functions, by property ($\frakF$-i).
We solve for the coefficient function $\hat{\psi}_{j_0,l_0}$ in terms of the
Herglotz operator:
\[
\hat{\psi}_{j_0,l_0}=\tfrac12\zeta\Hop_{\D}
\bigg[\frac{\frakF_{j_0-1,l_0+1,W}}{(4\hDelta R)^{\frac12}}\bigg].
\]
As for the first step $l_0=0$, the above formula applies in that case as well.
What is important is that then, the function $\frakF_{j_0-1,1,W}$ only depends 
data known at the beginning of Step 3, in view of ($\frakF$-i). By the 
correspondence \eqref{eq:index-set2.0} and the induction hypothesis, we have 
made sure that the system of equations \eqref{eq:main-flow} holds for all 
pairs $(j,l)\in\indsett_{2\kappa}$ with $l\ge1$ and $(j,l)\prec(j_0,1)$ 
as well as for $(j,l)$ with $0<j<j_0$ and $l=0$. This completes Step 3. 

\medskip

\noindent{\sc Step 4.} 
After having completed Step 3, we find ourselves in the following situation: 
the coefficient functions $b_j$ are known for $j< j_0$, while $\hat{\psi}_{j,l}$ 
are all known for $(j,l)\in\indsett_{2\kappa+1}$ with $(j,l)\prec(j_0+1,0)$. 
We proceed to determine $b_{j_0}$ using the equation \eqref{eq:Bj-det} with 
index $(j,l)=(j_0,0)$. That equation asserts that 
\[
2\Re b_{j_0}+\frakF_{j_0,0,W}=0\quad\text{on }\,\,\,\T,
\]
where $\frakF_{j_0,0,W}$ depends on the known data, by property ($\frakF$-ii). 
We solve for $b_{j_0}$ using the formula
\[
b_{j_0}=-\tfrac12\Hop_{\D_\e}[\frakF_{j_0,0}],
\]
and we get a function $b_{j_0}$ 
with $b_{j_0}(0)\in\R$. In view of Step 3, this choice makes sure that 
\eqref{eq:main-flow} holds for all pairs $(j,l)$ with $l=0$ and $0<j\le j_0$, 
as well as for all $(j,l)\in\indsett_{2\kappa}$ with $l\ge1$ and 
$(j,l)\prec(j_0,1)$. 
This completes Step 4, and we have extended the set of known data so 
that we may proceed to Step 3 with $j_0$ replaced with $j_0+1$.
Here, we should insert a word on smoothness. If $f$ is real-analytically 
smooth along $\T$, then $\Hop_\D[f]$ gets to be holomorphic in a larger disk 
$\D(0,\hdelta)$ for some $\hdelta>1$. Hence real-analytic smoothness carries 
over to the next step in the iterative procedure.
 
\medskip

The above algorithm continues until all the unknowns 
have been determined, up to the point where the whole index set
$\indsett_{2\kappa+1}$ has been exhausted. 
In the process, we have in fact solved the system of equations 
\eqref{eq:main-flow} for all indices $(j,l)\in \indsett_{2\kappa}$.
This means that if we form $h_s$ and $\psi_{s,t}$ in terms of the functions 
$b_j$ and $\hat{\psi}_{j,l}$ obtained with the above algorithm, and put 
$f_s=\exp(h_s)$, we find by exponentiating the Taylor series expansion 
of $\varpi_{s,t}$ in the parameters $s$ and $t$ that the flow equation 
\eqref{eq:flow-eq-W} holds. This completes the sketch of the proof of the 
lemma.
\end{proof}

\subsection{Orthogonal polynomial asymptotics for a 
general area form}

The results concerning root functions for more general area forms
apply also to the setting of orthogonal polynomials.
Although many things are pretty much the same, we make an effort to explain
what the precise result is in this context.

We need an appropriate notion of admissibility of the potential $Q$ which 
applies to the setting of orthogonal polynomials.
We recall that the spectral droplet $\calS_\tau$ is the contact set
\[
\calS_\tau=\{z\in\C:\,\hat Q_\tau(z)=Q(z)\},
\]
which is typically compact, where $\hat Q_\tau$ is the function
\[
\hat Q_\tau(z):=\sup\big\{q(z):\,q\in\mathrm{SH}(\C), \,\,q\le Q\,\,\,\text{on}
\,\,\,\C,\,\,\,q(z)\le\tau\log(|z|+1)+\Ordo(1)\big\}. 
\]

\begin{defn}
We say that the potential $Q$ is 
$\tau$-admissible if the following conditions are met:
\begin{enumerate}[(i)]
\item $Q:\C\to\R$ is $C^2$-smooth, 
\item $Q$ meets the growth bound
$$
\tau_Q:=\liminf_{|z|\to+\infty}\frac{Q(z)}{\log|z|}>\tau>0,
$$
\item The unbounded component $\Omega_\tau$ of the complement of the
spectral droplet $\calS_\tau$ is simply connected on the Riemann sphere 
$\hat{\C}$, with real-analytic Jordan curve boundary,
\item $Q$ is strictly subharmonic and real-analytically smooth in a 
neighborhood of the boundary $\partial\Omega_\tau$.
\end{enumerate}
\end{defn}

As before, we consider measures $\e^{-2mQ}V\diffA$, where the conformal factor
$V$ is assumed to be nonnegative, positive near the curve 
$\partial\Omega_\tau$, and real-analytically smooth in a neighborhood of 
$\partial\Omega_\tau$ with at most polynomial growth or decay at infinity 
\eqref{eq:polgrowthV}.
We denote by $\phi_\tau$ the surjective conformal mapping
\[
\phi_\tau:\Omega_\tau\to\D_\e,
\]
which preserves the point at infinity and has $\phi_\tau'(\infty)>0$. 
The function $\calQ_\tau$ is defined as the bounded 
holomorphic function on $\Omega_\tau$ whose real part equals $Q$ on 
the boundary $\partial\Omega_\tau$, and whose 
imaginary part vanishes at infinity.

The orthogonal polynomials $P_{m,n,V}$ have degree $n$, positive leading 
coefficient, and unit norm in 
$A^2_{mQ,V}$. They have the additional property that
\[
\langle P_{m,n,V},P_{m,n',V}\rangle_{mQ,V}=0,\qquad n\ne n'.
\]
We will work with $\tau=\frac{n}{m}$.

\begin{thm}\label{thm:twist-onp}
Suppose $Q$ is $\tau$-admissible for $\tau\in I_0$, where $I_0$ is a compact 
interval of the positive half-axis. Suppose in addition that $V$ meets the
above regularity requirements.
Given a positive integer $\kappa$ and a positive real $A$, there exists
a neighborhood $\Omega^{\circledast}_{\tau}$ of the closure
of $\Omega_{\tau}$ and bounded holomorphic functions $\calB_{j,\tau,V}$
on $\Omega^{\circledast}_{\tau}$, as well as 
domains $\Omega_{\tau,m}=\Omega_{\tau,m,\kappa,A}$ with 
$\Omega_{\tau}\subset\Omega_{\tau,m}\subset
\Omega^{\circledast}_{\tau}$ which meet
\[
\mathrm{dist}_{\C}(\Omega_{\tau,m}^c,
\Omega_{\tau})\ge Am^{-\frac12}(\log m)^{\frac12},
\]
such that the orthogonal polynomials enjoy the expansion
\[
P_{m,n,V}(z)
=m^{\frac14}(\phi_{\tau}'(z))^{\frac12}(\phi_{\tau}(z))^n
\e^{m\calQ_{\tau}}
\Big\{\sum_{j=0}^{\kappa}m^{-j}\calB_{j,\tau,V}(z)+
\Ordo\big(m^{-\kappa-1}\big)\Big\},
\]
on $\Omega_{\tau,m}$ as $n=\tau m\to+\infty$ while $\tau\in I_0$,
where the error term is uniform.
Here, the main term $\calB_{0,\tau,V}$ is zero-free and smooth up to the 
boundary on $\Omega_{\tau}$, positive at infinity, with prescribed modulus
\[
\lvert \calB_{0,\tau,V}(\zeta)\rvert=\pi^{-\frac14}[\hDelta Q(\zeta)]^{\frac14}
V(\zeta)^{-\frac12},\qquad\zeta\in\partial\Omega_\tau.
\]
\end{thm}

In view of Lemma~\ref{lem:main-flow-general}, the construction of 
approximately orthogonal quasi-polynomials may be carried out in the same way
as in the case when $V=1$. 

\subsection{$\bar\partial$-surgery in the context of 
more general area forms}
As for the $\bar\partial$-surgery technique in the presence of a 
conformal factor $V$,
we need to mention some modifications that are required. We recall 
that $V$ meets the
conditions \eqref{eq:polgrowthV} and \eqref{eq:subharm-V} and put 
\[
Q_m=Q-\frac{1}{2m}\log V.
\] 
We note that by \eqref{eq:subharm-V}, the modified potential $Q_m$ is 
subharmonic and that 
$Q_m$ is strictly subharmonic uniformly in a neighborhood of 
$\partial\Omega_\tau$, 
provided that $m$ is large enough. Here, $\Omega_\tau$ is the unbounded 
component of the off-spectral set $\calS_\tau^c$. 
Moreover, it follows from the growth and decay bounds 
\eqref{eq:polgrowthV} that $Q_m$ meets the required growth bound for 
large enough $m$.
We consider the solution $\hat{Q}_{\tau,m}$ to the obstacle problem
\[
\hat{Q}_{\tau,m}(z)=\sup\big\{u(z):u\in\mathrm{SH}_\tau(\C),\, u(z)\le 
Q_m(z)\text{ on }\C\big\}.
\]
The standard regularity theory applies, and implies that $\hat{Q}_{\tau,m}$ 
is $C^{1,1}$-smooth.
By our smoothness assumptions on $\calS_\tau$, as well as the assumed 
regularity of the potential $Q$ and the conformal factor $V$, the 
coincidence set
\[
\calS_{\tau,m}:=\big\{z\in\C: \hat{Q}_{\tau,m}(z)=Q(z)\big\}
\]
moves only very little away from $\calS_\tau$ for large $m$.
Indeed, by the main theorem of \cite{ss-quant}, the free boundary moves with
a smooth normal velocity under a smooth perturbation, 
and it follows that $\partial\calS_{\tau,m}$ 
is contained in a $\Ordo(m^{-1})$-neighborhood of $\partial\calS_\tau$.
We note that $\hat{Q}_{\tau,m}$ is automatically harmonic off $\calS_{\tau,m}$
and that $\hDelta \hat{Q}_{\tau,m}\ge 0$ holds generally.

We now apply H{\"o}rmander's classical $\bar\partial$-estimate with the 
obstacle solution $\hat{Q}_{\tau,m}$ as potential. 
Let $f\in L^\infty(\calS_{\tau,m})$, which vanishes whenever 
$\hDelta\hat{Q}_{\tau,m}=\hDelta Q_m=0$ on $\calS_{\tau,m}$. 
Then there exists a solution $u$ to the problem
\[
\bar\partial u=f
\]
which meets the $L^2$-bound
\[
\int_\C |u|^2\e^{-2m\hat{Q}_{\tau,m}}\diffA\le \frac{1}{2m}\int_{\calS_{\tau,m}}|f|^2
\frac{\e^{-2m Q}V}{\hDelta Q-\frac{1}{2m}\hDelta \log V}\diffA,
\]
provided that the right-hand side is finite.
Moreover, since $u$ is holomorphic off the compact set $\calS_{\tau,m}$,
this estimate implies a polynomial growth bound $u(z)=\Ordo(|z|^{n-1})$ as 
$|z|\to+\infty$.

After these modifications, the $\bar\partial$-surgery may be performed 
as in the earlier context of the planar area measure $\diffA$. For the 
details, we refer to Section~4.9 in \cite{HW} and the comments in 
Section~\ref{ss:d-bar} above.

\end{document}